\documentclass[letterpaper]{amsart}
\pdfoutput=1
\newtheorem{theorem}{Theorem}[section]
\newtheorem{lemma}[theorem]{Lemma}

\newtheorem{proposition}[theorem]{Proposition}
\newtheorem{corollary}[theorem]{Corollary}

\theoremstyle{definition}
\newtheorem{definition}[theorem]{Definition}

\theoremstyle{remark}
\newtheorem{remark}[theorem]{Remark}

\usepackage[pdftex]{graphicx}
\usepackage{color}
\usepackage{amsmath}
\usepackage{amsfonts}
\usepackage{amssymb}
\usepackage{epsfig}
\usepackage{rotating}
\usepackage{graphics}

\newcommand{\cC}{\mathcal{C}}

\newcommand{\cN}{\mathcal{N}}
\newcommand{\cH}{\mathcal{H}}
\renewcommand{\d}{\delta}

\newcommand{\be}{\begin{equation}}

\newcommand{\ee}{\end{equation}}

\newcommand{\om}{\omega}

\newcommand{\si}{\sigma}

\newcommand{\dz}{\wedge}

\newcommand{\ba}{\begin{array}}

\newcommand{\ea}{\end{array}}

\newcommand{\beq}{\begin{eqnarray}}

\newcommand{\eeq}{\end{eqnarray}}

\newtheorem{lm}{lemma}

\newtheorem{thee}{theorem}

\newtheorem{proo}{proposition}

\newtheorem{co}{corollary}

\newtheorem{rem}{remark}

\newtheorem{deff}{definition}

\newcommand{\bd}{\begin{deff}}

\newcommand{\ed}{\end{deff}}

\newcommand{\bl}{\begin{lm}}

\newcommand{\el}{\end{lm}}

\newcommand{\bp}{\begin{proo}}

\newcommand{\ep}{\end{proo}}

\newcommand{\bt}{\begin{thee}}

\newcommand{\et}{\end{thee}}

\newcommand{\bc}{\begin{co}}

\newcommand{\ec}{\end{co}}

\newcommand{\brm}{\begin{rem}}

\newcommand{\erm}{\end{rem}}

\newcommand{\der}{{\rm d}}

\usepackage{t1enc}
\def\frak{\mathfrak}

\def\Cal{\mathcal}

\newcommand{\newc}{\newcommand}

\newcommand{\id}{\operatorname{id}}

\newcommand{\rank}{\operatorname{rank}}
\newcommand{\Tor}{\operatorname{Tor}}
\newcommand{\Alt}{\operatorname{Alt}}

\let\ccdot\cdot
\def\cdot{\hbox to 2.5pt{\hss$\ccdot$\hss}}

\newc{\aR}{\mbox{\boldmath{$ R$}}}
\newc{\aS}{\mbox{\boldmath{$ S$}}}
\newc{\aT}{\mbox{\boldmath{$ T$}}}
\newc{\aW}{\mbox{\boldmath{$ W$}}}

\newc{\aK}{\mbox{\boldmath{$ K$}}}
\newc{\aL}{\mbox{\boldmath{$ L$}}}
\newcommand{\ce}{{\Cal E}}

\usepackage{amssymb}
\usepackage{amscd}

\newcommand{\Up}{\Upsilon}

\newcommand{\Ric}{\operatorname{Ric}}

\newcommand{\wh}{\widehat}
\let\i=\iota

\newcommand{\hook}{\raisebox{-0.35ex}{\makebox[0.6em][r]
{\scriptsize $-$}}\hspace{-0.15em}\raisebox{0.25ex}{\makebox[0.4em][l]{\tiny
 $|$}}}

\def\LOT{{\rm LOT}}

\newcommand{\cW}{{\Cal W}}

\newcommand{\bma}{\begin{pmatrix}}
\newcommand{\ema}{\end{pmatrix}}

\newcommand{\IT}[1]{{\rm(}{\it{#1}}{\rm)}}

\newcommand{\gac}[1]{{\stackrel{\scriptscriptstyle{#1}}{\Gamma}}\phantom{}}
\newcommand{\nn}[1]{(\ref{#1})}

\newcommand{\bg}{\mbox{\boldmath{$ g$}}}
\newcommand{\G}[1]{{\stackrel{\scriptscriptstyle{#1}}{G}}\phantom{}}
\newcommand{\A}[1]{{\stackrel{\scriptscriptstyle{#1}}{A}}\phantom{}}

\newc{\obstrn}[2]{B^{#1}_{#2}}

\newcommand{\rpl}                         % +) or <+
{\mbox{$
\begin{picture}(12.7,8)(-.5,-1)
\put(0,0.2){$+$}
\put(4.2,2.8){\oval(8,8)[r]}
\end{picture}$}}

\newcommand{\lpl}                         % (+ or +>
{\mbox{$
\begin{picture}(12.7,8)(-.5,-1)
\put(2,0.2){$+$}
\put(6.2,2.8){\oval(8,8)[l]}
\end{picture}$}}

\usepackage{ifthen}

\newc{\tensor}[1]{#1}
\newc{\Mvariable}[1]{\mbox{#1}}
\newc{\down}[1]{{}_{#1}}
\newc{\up}[1]{{}^{#1}}

\newc{\JulyStrut}{\rule{0mm}{6mm}}
\newc{\midtenPan}{\mbox{\sf S}}
\newc{\midten}{\mbox{\sf T}}
\newc{\midtenEi}{\mbox{\sf U}}
\newc{\ATen}{\mbox{\sf E}}
\newc{\BTen}{\mbox{\sf F}}
\newc{\CTen}{\mbox{\sf G}}

\def\sideremark#1{\ifvmode\leavevmode\fi\vadjust{\vbox to0pt{\vss% the remark
 \hbox to 0pt{\hskip\hsize\hskip1em%                          will appear only
 \vbox{\hsize3cm\tiny\raggedright\pretolerance10000%          on the side
 \noindent #1\hfill}\hss}\vbox to8pt{\vfil}\vss}}}%
\numberwithin{equation}{section}

\newcounter{romenumi}
\newcommand{\labelromenumi}{(\roman{romenumi})}

\begin{document}
\title[Calculus and invariants on almost complex manifolds]{Calculus and invariants on almost complex manifolds, including projective and conformal geometry}
\vskip 1.truecm \author{A.\ Rod Gover} \address{Department of Mathematics,
  University of Auckland, Private Bag 92019, Auckland, New Zealand; Mathematical Sciences Institute,
Australian National University, ACT 0200, Australia}
\email{rgover@auckland.ac.nz} \author{Pawe\l~ Nurowski}
\address{Instytut Fizyki Teoretycznej, Uniwersytet Warszawski,
  ul. Hoza 69, Warszawa, Poland} \email{nurowski@fuw.edu.pl}
\thanks{This research was supported by the Royal Society of New
  Zealand via Marsden Grant 10-UOA-113, and by the Polish Ministry of
  Research and Higher Education under grants NN201 607540 and NN202
  104838}

\subjclass[2000]{53C15, 53C05, 53A20, 53A30, 53C25}  \keywords{almost
  complex, connections, special geometric structures, projective
  differential geometry, conformal differential geometry}

%\date{\today}

\begin{abstract}
We construct a family of canonical connections and surrounding
basic theory for almost complex manifolds that are equipped with an
affine connection. This framework provides a uniform approach to
treating a range of geometries. In particular we are able to construct
an invariant and efficient calculus for conformal almost Hermitian
geometries, and also for almost complex structures that are equipped
with a projective structure. In the latter case we find a projectively
invariant tensor the vanishing of which is necessary and sufficient
for the existence of an almost complex connection compatible with the
path structure. In both the conformal and projective setting we give
torsion characterisations of the canonical connections and introduce
certain interesting higher order invariants.
\end{abstract}
\maketitle
%*************
\tableofcontents
\newcommand{\bbS}{\mathbb{S}}
\newcommand{\bbT}{\mathbb{T}}
\newcommand{\bbR}{\mathbb{R}}
\newcommand{\sog}{\mathbf{SO}}
\newcommand{\slg}{\mathbf{SL}}
\newcommand{\og}{\mathbf{O}}
\newcommand{\soa}{\frak{so}}
\newcommand{\sla}{\frak{sl}}
\newcommand{\sua}{\frak{su}}
\newcommand{\ua}{\frak{u}}
\newcommand{\ug}{\mathbf{U}}
\newcommand{\dr}{\mathrm{d}}
\newcommand{\sug}{\mathbf{SU}}
\newcommand{\glg}{\mathbf{GL}}
\newcommand{\gat}{\tilde{\gamma}}
\newcommand{\Gat}{\tilde{\Gamma}}
\newcommand{\thet}{\tilde{\theta}}
\newcommand{\Thet}{\tilde{T}}
\newcommand{\rt}{\tilde{r}}
\newcommand{\st}{\sqrt{3}}
\newcommand{\kat}{\tilde{\kappa}}
\newcommand{\kz}{{K^{{~}^{\hskip-3.1mm\circ}}}}
\newcommand{\bv}{{\bf v}}
\newcommand{\di}{{\rm div}}
\newcommand{\curl}{{\rm curl}}
\newcommand{\cs}{(M,{\rm T}^{1,0})}
\newcommand{\tn}{{\mathcal N}}

%*************

\newcommand{\lccon}{{\stackrel{\scriptscriptstyle{LC}}{\nabla}}\phantom{}}

\newcommand{\con}[1]{{\stackrel{\scriptscriptstyle{#1}}{\nabla}}\phantom{}}
\newcommand{\stack}[2]{{\stackrel{\scriptscriptstyle{#1}}{#2}}\phantom{}}
\newcommand{\ostack}[2]{{\stackrel{#1}{\scriptscriptstyle{#2}}}\phantom{}}

\newcommand{\dl}{{\stackrel{\scriptscriptstyle{LC}}{D}}\phantom{}}

\section{Introduction}

Let $M$ be a smooth manifold of even dimension $n=2m$. An almost
complex structure (ACS) $J$ on $M$ is an endomorphism of the tangent
bundle $TM$ such that $J^2=-1$. The study of almost complex structures
has a rich history, especially in connection with complex geometry
such as the theory of K\"ahler manifolds and closely linked
themes. There is by now rather sophisticated machinery available for
the treatment of almost complex geometries
\cite{Bismut,FFS,Gray1,GrayH,L13,PA1,PA2,PA3,TV}. However the basic calculus
has been typically developed starting from the assumption that there
is an almost Hermitian metric given as part of the data. From there it
is often not clear what parts of the results may be applied to
different geometric structures, or in more general settings.

Our aim in this article is to develop a uniform approach to the
calculus for almost complex manifolds which are also equipped with
some additional geometric structure such as a conformal structure, or
a projective structure. We indicate how this may be applied to the
construction of invariants of the structure; we treat in more detail
some of the less obvious new invariants that are seen to arise
naturally from this perspective. Because of the nature of our
endeavour there are inevitable close links with many results in the
literature, especially in the case where we specialise to almost
Hermitian geometry.  Within the scope of this article it would be
impossible to do justice to the very nice work that has been done in
this direction by many authors. However the works of Libermann, Obata,
and Lichnerowicz \cite{Lib10,L13,obata} are particularly
relevant. Much of that work is put into a uniform context by Gauduchon
in \cite{GauHconn}, where also some extensions and Dirac operators are
discussed.

Briefly the treatment and strategy for the calculus development is as
follows. In Section \ref{accalc} we treat almost complex affine
manifolds. This means the data of the structure is an almost complex
manifold $M$, of any even dimension, equipped with an almost complex
strucure $J$ and an affine connection $\nabla$. For this setting we
develop a basic calculus that includes a family of connections
determined by $(M,J,\nabla)$ that are canonical and almost complex,
meaning that they preserve $J$. The point here is that this is
developed in such a way that it then easily specialises to a range of
other geometries where there is additional structure, and so provides
a treatment of these that is uniform. The structures we treat are
almost Hermitian geometry, conformal almost Hermitian geometry, and
finally projective almost complex geometry.  A conformal almost
Hermitian geometry is the structure given by $(M^{2m},J,c)$ where
$m\geq 2$, and $c$ is a conformal equivalence class of almost
Hermitian metrics. A projective almost complex geometry consists of
$(M^{2m},J,p)$ where $m\geq 1$, and $p$ is a projective equivalence
class of torsion-free connections; two connections
$\nabla$ and $\widehat{\nabla}$ are said to be projectively equivalent
if they have the same geodesics as unparameterised curves.  A key
point, for our development, is that each of these structures can be
shown to have a canonical affine connection and so using this one may
immediately employ the general machinery developed in Section
\ref{accalc}.
Now we outline in more detail the developments and some of the main results.

As mentioned Section \ref{accalc} develops the basic tangent bundle
calculus for general affine almost complex manifolds. We prove that
the affine connection $\nabla$ and $J$ determine a fundamental
$(1,2)$-tensor $G$ that plays a central role throughout the
article. Using this we prove, for example, in Proposition \ref{gencon}
that the structure determines a 1-parameter family of canonical
connections on $TM$ that preserve the almost complex structure $J$.
There is a distinguished connection $\con{KN}$ in the class with
anti-Hermitian torsion. This has the property that in the case that
$\nabla$ is torsion free then the torsion of $\con{KN}$ is precisely
the Nijenhuis tensor, see Proposition \ref{KN}.  Section \ref{accalc}
also contains many technical results for use later in the article. For
example we introduce there an important notion of compatibility
between an affine connection and $J$, this amounts to $G$ being
completely trace-free in the sense of Lemma \ref{compl}.

Next Section \ref{AHs} treats the case of almost Hermitian
geometry. Of course on almost Hermitian manifolds the basic tensor
calculus has been treated considerably in the literature. So the main
points of this section are first: to indicate how the usual objects
arise by simple specialisation of the tools from the $G$-calculus of
Section \ref{accalc}; and second to lay out the almost Hermitian
results for comparison with the results for conformal and projective
structures which follow in the later sections. Building on work of
Gilkey \cite{Gilkey}, and others, there is a classification by
Gray-Hervella \cite{GrayH} of almost Hermitian manifolds according to
a $U(m)$-decomposition of $\nabla \om$. We describe in Proposition
\ref{cin} how certain key cases from the Gray-Hervella list, such as
nearly K\"ahler, Hermitian, and almost K\"ahler, are identified in
terms of $G$.  Beginning with the Levi-Civita connection $\nabla$
then from the family of almost complex connections of Proposition
\ref{gencon} there is a unique connection that preserves the metric.
We show in Theorem \ref{char} how, in our framework, there is a
torsion characterisation of this distinguished connection. This (or
its equivalents in the literature) provides a universal solution to
the problem of finding a type of characteristic connection for each of
the structures of the Gray-Hervella classification and this is the
subject of Corollary \ref{charGH}.

Section \ref{confS} begins the more involved application of the
approach. On a conformal almost Hermitian manifold there is a
canonical and unique Weyl connection $\con{c}$ that is compatible with
$J$ (cf.\ \cite{Vaisman,Bailey}). This
provides the basic input to generate the conformal version of the
tools from Section \ref{accalc}.  Using this one concludes there is a
conformally invariant connection $\con{gc}$ determined by structure
$(M,J,c)$ which preserves $J$ and the conformal structure, see
Proposition \ref{generc}.  Theorem \ref{torsionB} shows that
connections with torsion that both, preserve the conformal structure
and are suitably compatible with the complex structure, are
parametrised by their torsions. This is then used to show that the
vanishing of a canonical conformal torsion invariant suffices to
characterise the connection $\con{gc}$ among all almost complex
connections on the structure $(M,J,c)$, see Proposition \ref{confTs}
and Theorem \ref{cchar}. This means that the in conformal setting the
results are really as strong as in the almost Hermitian case, which we
find surprising. In the subsections \ref{NKWS}, \ref{LCAKS}, and
\ref{cGH} we show how the structures of the Gray-Hervella conformal
almost Hermitian classification are described and treated via the
$G$-calculus. Finally in Section \ref{ci} we show that the canonical Weyl
structure of the manifold $(M,J,c)$ leads to some interesting higher
order conformal invariants, including global invariants and objects
that are analogues of $Q$-curvature.

Section \ref{PGS} is the last of the theoretical developments and
treats almost complex manifolds that are also equipped with a
projective structure $p$. In analogy with the conformal case, we prove
in Proposition \ref{projJ} that there is a unique connection $\con{p}$
in $p$ that is compatible with $J$. We observe in Corollary
\ref{pgeod} that this implies a distinguished class of parametrised
curves, namely those curves which are the geodesics of the connection
$\con{p}$. (For an affine connection $\nabla$ its {\em geodesics} are
those curves whose tangent field $X$ satisfies $\nabla_X X=0$.)
Theorem \ref{tprop} determines a distinguished almost complex
connection $\con{JP}$. This has the property that if a certain
projective invariant $\G{p}_-^{\rm symm}$ vanishes then $\con{JP}$ has
as geodesics the mentioned distinguished class of curves, see
Corollary \ref{punch}.  ($\G{p}_-^{\rm symm}$ it is the anti-Hermitian
symmetric part of the fundamental $G$-invariant $\G{p}$ for the
structure $(M,J,p)$ see e.g. Subsection \ref{pGH}.)  The main result
of Section \ref{PGS} is Theorem \ref{projR} which proves that if
$\nabla'$ is any affine connection, that preserves $J$ and agrees
with the path structure $p$, then necessarily this invariant
$\G{p}_-^{\rm symm}$ vanishes identically, $\nabla'$ is simply related
to $\con{JP}$, and the geodesics of $\nabla' $ agree with the
distinguished curves. This is a fundamental result concerning the
relation between almost complex and projective geometry.  Among other
things this shows that the connection $\con{JP}$ is optimal and that
the condition of vanishing of $\G{p}_-^{\rm symm}$ is an important and
canonical condition of {\em compatibility} between the complex
structure $J$ and the projective structure $p$ (hence the Definition
\ref{PJcomp}).  There is a torsion characterisation of $\con{JP}$
given in Corollary \ref{torch}, so for the compatible projective
almost complex structures the results are again as strong as for the
almost Hermitian case. In Subsection \ref{pGH} we describe projective
analogues of the objects in the Gray-Hervella conformal classification
and discuss related issues. Finally in the Section higher projective
invariants are discussed briefly in Subsection \ref{pinvt}.

Section \ref{EXS} shows that examples are available for the various
structures. In fact we treat just a few cases here as for most of the
structures it is rather obvious that there will be structures
available satisfying the various conditions.

\subsection{Conventions}
 For simplicity all structures will be assumed smooth, meaning
 $C^\infty$.  Unless otherwise stated, $X,Y$ denote arbitrary sections
 of the tangent bundle $TM$. We also from time to time, as convenient,
 employ Penrose's abstract index formalism \cite{ot}. For example
 $\ce^a$ is an alternative notation for $TM$ (or its section space, we
 shall not distinguish) and $X^a,Y^a$ denote sections thereof. The
 almost complex structure $J$ is written via abstract indices as
 $J^a{}_b$, so that $J X$ may be written $J^a{}_b X^b$, and $J^a{}_b
 J^b{}_c=-\delta^a_c$, where $\delta^a_b$ is the pointwise identity
 endomorphism on $TM$. Here the repeated indices indicate
 contractions.

\section{Calculus on an almost complex affine manifold}\label{accalc}

Here we develop a canonical calculus for almost complex manifolds that
are also equipped with an affine connection. This then forms the basis
for our subsequent treatment of other geometries.

\subsection{Almost complex affine connections} \label{acac}

Given an almost complex manifold $(M,J)$ an affine connection $\nabla$
on $M$ is called {\em almost complex} if it preserves $J$. We first
observe that any affine connection can be modified to yield such a
connection.

\newcommand{\nag}{\nabla^{\vspace*{-6pt} G}}

 Let $\nabla$ be any affine connection on an almost complex
  $n$-manifold $M$. Let $H$ be a $(1,2)$ tensor field on $M$ and
consider the connection $\con{H}$ defined by 
$$
\con{H}_X Y: = \nabla_X Y+ H(Y,X) \quad \mbox{for} \quad X,Y\in \Gamma(TM).
$$
We seek $H$ such that $\con{H}_X (JY)= J\con{H}_X Y$. This is
equivalent to
\begin{equation}\label{jh}
\nabla_X (JY) - J\nabla_X Y =JH(Y,X)-H(JY,X).
\end{equation}
Evidently if $H$ is a solution then we obtain another solution by
adding a $(1,2)$ tensor field $K$ which is $J$ linear in the first
argument: $K(JY,X)=JK(Y,X)$. To remove this freedom we replace $H$
with a $(1,2)$ tensor $G$ that is assumed to be $J$-antilinear in the
first argument: $G(JY,X)=-JG(Y,X)$.  Then \nn{jh} becomes
$$
\nabla_X (JY) - J\nabla_X Y =2 JG(Y,X) ,
$$
which may be solved for $G$ to yield $G(Y,X)=-\frac{1}{2}J(\nabla_X
J)Y$.  Moreover for any $H$ solving \nn{jh}, $G$ is the complex
anti-linear part over the first argument, i.e.\
$G(\cdot,\cdot)=\frac{1}{2}(H(\cdot,\cdot)+JH(J\cdot,\cdot))$. Thus
the general solution to \nn{jh} is of the form $H=G+K$, for some $K$ as above.
We summarise as follows.
\begin{proposition}\label{gener} 
  Let $(M,J)$ be an almost complex $n$-manifold and $\nabla$ an arbitrary
 affine connection on $M$. 
Then $\con{G}$ is an almost complex connection, that is
 $\con{G} J=0$, where 
$$
\con{G}_X Y:= \nabla_X Y + G(Y,X),
$$
and
\begin{equation}\label{Gform}
G(X,Y): =\frac{1}{2}(\nabla_Y J)JX=-\frac{1}{2}J(\nabla_Y J)X .
\end{equation}

Moreover, if $\con{H}$ is any connection preserving $J$ then
$$
\con{H}_X Y =\con{G}_X Y + K(Y,X)
$$
where $K$ is a $(1,2)$ tensor which is complex linear in the first
argument.
\end{proposition}

In places below it is convenient to use abstract index notation for
$G$ and the related connections: we write $G^{a}{}_{bc}Y^bX^c$ for the
vector field $G(Y,X)$ and
$$
\con{G}_a Y^b:= \nabla_a Y^b + G^b{}_{ca}Y^c .
$$
In this notation 
$$
G^b{}_{ca}:= \frac{1}{2}(\nabla_a J^b{}_d)J^d{}_c= -\frac{1}{2}J^b{}_d \nabla_aJ^d{}_c~. 
$$

We note some properties of $G$ for later use.
\begin{lemma}\label{tffs} For all tangent vectors $X$,
$G(\cdot,X)$ is trace-free and $JG(\cdot ,X)$ is trace-free.
If $\nabla $ preserves a volume form on M, then  $\con{G}$ preserves the same volume form.
\end{lemma}
\begin{proof}
We have 
$$
2G^b{}_{ba}= (\nabla_a J^b{}_d)J^d{}_b=\frac{1}{2}\nabla_a
(J^b{}_dJ^d{}_b) =0.
$$
This proves first claim and the last statement follows immediately.

For the second claim we re-express $2J^a{}_b G^b{}_{cd}$. This is 
\begin{equation}\label{JG}
J^a{}_b(\nabla_d J^b{}_e)J^e{}_c = \nabla_d J^a{}_c.
\end{equation}
Since the almost complex structure $J$ is trace-free the result follows.
\end{proof}
\begin{lemma}\label{Gprop}
  $ G(JX,Y)=-JG(X,Y) $ and hence 
its Hermitian part
$$
G_+(X,Y):=\tfrac{1}{2}\big(G(X,Y)+G(JX,JY)\big)
$$ 
and its anti-Hermitian part 
$$
G_-(X,Y):=\tfrac{1}{2}\big(G(X,Y)-G(JX,JY)\big)
$$ 
have the following  properties: 
\begin{equation}\label{Jlin}
G_\pm (X,JY)=\pm JG_\pm (X,Y),
\end{equation}
and
\begin{equation}\label{Jlin2}
G_\pm (JX,Y)= - J G_\pm (X,Y)
\end{equation}
\end{lemma}
\begin{proof} The property $ G(JX,Y)=-JG(X,Y) $ is a consequence of
  the proof of Proposition \ref{gener}; alternatively it is easily
  verified from \nn{Gform}.
Using this we have
$$
\begin{array}{rcl}
2G_{\pm}(X,JY) &=& G(X,JY)\pm JG(X,Y) \\
               &= &J\big( \pm G(X,Y)- JG(X,JY) \big)\\
&= & J\big( \pm G(X,Y)+ G(JX,JY) \big)\\
&= & \pm 2 J G_{\pm}(X,Y)~.
\end{array}
$$
A similar calculation yields \nn{Jlin2}.
\end{proof}

In Proposition \ref{gener} we showed that on an almost complex
manifold $(M,J)$ the space of almost complex connections is affine
modelled on the space of $(1,2)$ tensor fields which are complex linear in
the first argument.  Fix an affine connection $\nabla$, as in that Proposition.
It follows immediately from the property
\nn{Jlin} of $G_+$, in the Lemma \ref{Gprop}, that we may, in
particular, use multiples of $G_+$ to modify the connection $\con{G}$,
while retaining the property that the new connection preserves $J$.
\begin{proposition} \label{gencon} Let $\nabla$ 
be an affine connection on an almost complex manifold
$(M,J)$, and $\con{G}$ the corresponding almost complex connection as
above. For any $t \in \mathbb{R}$, $\con{t}$ is an  
almost complex connection where 
$$
\con{t}_X Y: = \con{G}_XY + t G_+(X,Y).
$$
\end{proposition}

\subsection{Torsion and Integrability} \label{int} 
In the above we have not made any assumptions concerning the torsion
of $\nabla$. Beginning with any connection $\tilde{\nabla}$ with
torsion $\tilde{T}$ the related connection $\nabla$ defined by
$\nabla_XY:= \tilde{\nabla}_XY-\frac{1}{2}\tilde{T}(X,Y)$ is torsion
free. In the case
that $\nabla$ is torsion free then we obtain very useful formulae for the Nijenhuis tensor
$N_J$. (The normalisation of $N_J$ is for convenience.)
\begin{proposition} \label{nj} For any torsion free connection
  $\nabla$ on an almost complex manifold $(M,J)$ we have
\begin{equation}\label{njform}
4 N_J(X,Y)=(\nabla_X J)J Y-(\nabla_Y J)J X 
+(\nabla_{JX} J) Y-(\nabla_{JY} J)X ,
\end{equation}
and hence
\begin{equation} \label{njformG}
N_J(X,Y)= G_-(Y,X)-G_-(X,Y)
\end{equation}
\end{proposition}
\noindent{\bf Proof:} The formula \nn{njform} is well known, see e.g.\
\cite{Kruglikov}. Using \nn{Gform} to rewrite this
in terms of $G$ yields \nn{njformG}.  \quad $\Box$

It is immediate from the definition of $\con{G}$ that if $\nabla$ is a
torsion free connection then
\begin{equation}\label{Ttf}
T^G(X,Y)=G(Y,X)-G(X,Y).
\end{equation}
Thus we obtain the following consequence of the Proposition \ref{nj}.
\begin{corollary}\label{torG}
  Let $\nabla$ be a torsion free connection on an almost
  complex manifold $(M,J)$, and $\con{G}$ the corresponding almost
  complex connection given by Proposition \ref{gener}. Then the
  anti-Hermitian part of the torsion of $\con{G}$ is the Nijenhuis tensor,
\begin{equation}\label{Tory}
T^G_-(X,Y)=N_J(X,Y) .
\end{equation}
\end{corollary}
\noindent{\bf Proof:} 
Taking the anti-Hermitian part of \nn{Ttf} gives the result by
\nn{njformG}. \quad $\Box$

Evidently we may use these observations to select a distinguished
connection $ \con{KN} $ from the class given in 
 Proposition
\ref{gencon}:
\begin{equation}\label{KNdef}
\con{KN}_X Y= \con{G}_X Y + G_+(X,Y). 
\end{equation}
For this connection, the Hermitian part of the torsion is zero, while
the anti-Hermitian part of its torsion agrees with the anti-Hermitian
part of the $\con{G}$ torsion.
From this observation and Proposition \ref{nj} we have the following result.
\begin{proposition}\label{KN}  Beginning with any affine connection $\nabla$, the associated connection $\con{KN}$ 
has anti-Hermitian torsion $\Tor(\con{KN})$.  In the case that
$\nabla$ is torsion free, we have $\Tor(\con{KN})=N_J$, the Nijenhuis
tensor.
\end{proposition}
\begin{remark}\label{KNR} The connection $\con{KN}$ is readily verified to
be the classical connection of \cite[Theorem 3.4, Section IX]{KN},
where the property of its torsion is also noted. Note that an
immediate corollary of Proposition \ref{KN} is that an almost complex
manifold admits an almost complex torsion free connection if and only
if $J$ is integrable.
\end{remark}

\subsection{Compatible affine connections}\label{comS}
Let $(M,J)$ be an almost complex structure. An affine connection
$\nabla$ on $M$ will be said to be {\em compatible} with $J$ if
\begin{equation}\label{com}
\nabla_a J^a{}_b=0.
\end{equation}
\begin{lemma}\label{compl}
    On an almost complex structure $(M,J)$ an affine connection
    $\nabla$ is compatible if and only if $G^{\nabla}$ is trace-free;
    equivalently, if and only if $JG^{\nabla}$ is trace-free;
    equivalently, if and only if $G^{\nabla}(\cdot,J\cdot)$ is
    trace-free; equivalently, if and only if
    $G^{\nabla}(J\cdot,J\cdot)$ is trace-free.
\end{lemma}
\noindent{\bf Proof:} Let $\nabla$ be any affine connection.  We have
already in Lemma \ref{tffs} that in any case the first trace of $G$
and the first trace of $JG$ both vanish (i.e. this feature of $G$ is
not related to compatibility).

Recall
$G^b{}_{ca}:= \frac{1}{2}(\nabla_a J^b{}_d)J^d{}_c$ so 
$$
2G^b{}_{cb}= (\nabla_b J^b{}_d)J^d{}_c,
$$ 
which is clearly zero if and only if 
$ \nabla_b J^b{}_d=0$. 

On the other hand recall from Lemma \ref{tffs} that $2JG$ is
\begin{equation}\label{JG}
J^a{}_b(\nabla_d J^b{}_e)J^e{}_c = \nabla_d J^a{}_c
\end{equation}
and so the only available trace yields $\nabla_a J^a{}_b$. 

The trace of $2G(X,J \cdot )$ is $X^b \nabla_aJ^a{}_b$, so this case
is also clear. Then the final statement thus follows.
\quad $\Box$

\section{Almost Hermitian geometry} \label{AHs} \label{AHs} 

Let
$(M^n,J)$ be an almost complex manifold of dimension $n\geq 4$, and $g$
a Riemannian metric on $M$.  The triple $(M,J,g)$ is said to be {\em
  almost Hermitian} if $J$ is orthogonal with respect to $g$, that is
$$
g(J X, JY)=g(X,Y)
$$
for all tangent vector fields $X,Y$.
\medskip
\begin{remark}\label{av}
Note that if $g$ is any Riemannian metric on $(M,J)$, then the 
 Hermitian part of $g$, that is 
$$
g_+(X,Y)=\frac{1}{2}(g(X,Y)+g(JX,JY)),
$$
is positive definite, and so 
$(M,J,g_+)$ is
an almost Hermitian structure.
\end{remark}
Henceforth in this section we shall assume $(M,J,g)$ is an almost
Hermitian (AH) structure. In this setting we also have the skew symmetric
K\"ahler form
\begin{equation} \label{cc3}
\omega(X,Y):=g(X,JY).
\end{equation}

Let us make two comments regarding this section. First the results in
this section are for the most part well known.  Nevertheless we want
to understand some of the standard structures from almost Hermitian
geometry in terms of the $G$-calculus developed above.  This serves to
put our discussion in context and gives us a basis from which we may
compare the conformal and projective treatments in the next sections.
The second point is that for the reason that the material is known we
are brief here and some of the key results we shall use are drawn from
the later Section \ref{confS}; the point is that there the results are
proved in a broader context.

Proceeding now, in this section we shall use $G$ to denote the tensor
of \nn{Gform} where $\nabla=\con{LC}$ is the Levi-Civita connection of
$g$. With this specialisation $G$ is what in the literature is an
example of an intrinsic torsion and $\con{G}$ is then what is usually
called the {\em canonical Hermitian connection}, see
e.g. \cite{CSal,PA2}. This classical object is the {\em first
  canonical connection} of \cite{L13} and to the best of our knowledge
originated in the work \cite{Lib10} of Libermann. In fact the latter
source gives a 1-parameter family of canonical almost Hermitian
connections and this family is discussed in detail in \cite{GauHconn}
where it explained how the various connections of \cite{Bismut,L13},
as well as a torsion minimising connection introduced in
\cite{GauHconn}, arise from Libermann's family $\nabla^t$; the
connection $\con{G}$ of this section is the operator $\nabla^0$ from
there.  Such a family arises because the almost Hermitian metric
enables a finer decomposition of $G$ than is available in Section
\ref{accalc} above. Nevertheless we shall not explore that here, since
without geometry specific refinement the $G$-calculus from above is
both simple and universally applicable, and these are the features that
we apply in the later sections.

Let us  write $\G{}(X,Y,Z):= g(X,\G{}(Y,Z))$, and $\G{}_{\pm}(X,Y,Z):=
g(X,\G{}_{\pm}(Y,Z))$.
First we introduce some general facts that we shall use. From  Proposition
\ref{fGprop} (below) we have that 
\begin{equation}\label{Gsk}
\G{}(X,Y,Z)=-\G{}(Y,X,Z)
\end{equation} 
and so $\con{G}$ is a metric connection. 
The same Proposition also proves that 
\begin{equation}\label{prem} 
G_-(X,Y,Z)=-G_-(Y,X,Z).
\end{equation}
Also from there, or alternatively from the skew symmetry of $\om$, it follows
that on an AH structure $g(\cdot,JG(\cdot,\cdot))$ is also skew over first
and second arguments, that is
\begin{equation}\label{skewp}
g(X,JG(Y,\cdot))+g(Y,JG(X,\cdot))=0 \Leftrightarrow G(X,JY,Z)+G(Y,JX,Z)=0,
\end{equation}
where the equivalence uses Lemma \ref{Gprop}. Together  \nn{Gsk} and
\nn{skewp}  imply that $G(X,Y,Z)$ is anti-Hermitian in the
first pair. That is
\begin{equation}\label{Ganti}
G(X,Y,Z)=-G(JX,JY,Z).
\end{equation}

Now recall that an AH structure $(M,J,g)$ is said to be a {\em Hermitian}
structure if $N_J=0$. Thus from Proposition \ref{nj} 
$G_-(X,Y,Z)=G_-(X,Z,Y)$. But this with \nn{prem} implies that $G_-=0$.
Conversely, again using Proposition \nn{nj}, $G_-=0$ implies $N_J=0$. Thus $g$ is Hermitian if and only if $G_-=0$.

An AH structure that satisfies 
\begin{equation}\label{cosym}
\delta \omega=0
\end{equation}
 is said to be {\em semi-K\"ahler} (or {\em co-symplectic}). Here
 $\delta$ is the formal adjoint of the exterior derivative; so note
 that this condition \nn{cosym} is precisely that the Levi-Civita
 connection is compatible with $J$, as in the definition \nn{com}.

A stronger condition on an AH manifold $(M,J,g)$ is given by
$$
(\lccon_XJ)X=0,\quad\quad\forall X\in \Gamma(TM)
$$ and this defines structures that are called {\em nearly K\"ahler}.
(In dimension $n=4$ this is equivalent to K\"ahler, as below, but in
higher dimensions it is a strictly weaker condition.)  This condition
is obviously that same as requiring that $G$ be anti-symmetric:
$G(X,Y)=-G(Y,X)$.

Next an AH structure is said to be {\em almost K\"ahler}
(or {\em symplectic}) if
$$
\der\omega =0. 
$$
This is easily re-written directly in terms of $G$:
\begin{equation}\label{dw}
\der\omega =0 \quad \Leftrightarrow \quad \Alt_{(X,Y,Z)} g(X,JG(Y,Z)) =0,
\end{equation}
where $\Alt$ is the projection to
the completely  skew part.

An AH structure is called {\em K\"ahler} if we have the two conditions 
\begin{equation}\label{Kc}
\der\omega =0\quad\quad{\rm and}\quad\quad N_J = 0.
\end{equation}
The K\"ahler condition, when expressed in terms of the
Levi-Civita connection $\lccon$ of the metric $g$, may be expressed:
$$
\lccon_XJ=0,\quad\quad \forall X\in {\rm T}M.
$$ Using the machinery here, this well known characterisation is
easily recovered as follows. First if $J$ is parallel for the
Levi-Civita connection then it follows at once that $\omega$ is
parallel, and thus $d\omega=0$ since the Levi-Civita connection is
torsion free. On the other hand from Proposition \ref{KN} we also have
that $N_J=0$ (see Remark \ref{KNR}). For the other direction suppose
that the conditions \nn{Kc} hold. Since $G_-=0$ it follows that $G(X,J
Y,Z)$ is Hermitian on the argument pair $Y,Z$ (i.e.\ $G(X,J
Y,Z)=G(X,J^2 Y,J Z) $). But using \nn{dw} it follows that $G(X,J
Y,Z)=G(Z,JY,X)-G(Z,JX,Y) $, and so $G(X,JY,Z)$ is also Hermitian on
the argument pair $X,Y$. But comparing with \nn{Ganti} it then follows that 
 $G=0$.

 To summarise, we have the following.
\begin{proposition}\label{cin}
If $\nabla=\lccon$ is the Levi-Civita connection of an almost
Hermitian manifold $(M,J,g)$ then the structure is:
\begin{itemize}
\item[H)]  Hermitian iff 
$$
G_-(X,Y)=0;
$$
\item[NK)]  nearly K\"ahler iff 
$$
G(X,Y)+G(Y,X)=0;
$$
\item[AK)]  almost K\"ahler iff
$$
\Alt_{(X,Y,Z)} g(X,JG(Y,Z)) =0;
$$
\item[K)]  K\"ahler iff 
$$G= 0.$$
\end{itemize} 
\end{proposition}

The Gray-Hervella classification of AH structures \cite{GrayH} is based
around the $U(m)$ decomposition of $\nabla \omega$, where $\nabla=\lccon$. But
$$
\nabla_a \om_{bc}=g_{be}\nabla_aJ^e{}_c=2g_{be}J^e{}_bG^b{}_{ca}
$$
or equivalently  
\begin{equation}\label{Gvs}
\nabla_X \om (Y,Z)= 2 g(Y,JG(Z,X)).
\end{equation}
Thus the Gray-Hervella classification could equivalently be formulated
as a $U(m)$ decomposition of the $G$ for the Levi-Civita
connection. In the almost Hermitian setting $G$ has a number of
additional symmetries and properties (some mentioned in \nn{Gsk} to
\nn{Ganti} above) that simplify the situation considerably.  This
leads to the next observation.
\begin{proposition}\label{NKsk}
An AH structure is nearly K\"ahler if and only if  
$$
G(\cdot,\cdot,\cdot):= g(\cdot ,G(\cdot,\cdot))
$$ is completely alternating.
If this holds then 
$$   
G_+=0 
$$
while
$$ 
\Tor(\con{G})=-2 G_- =N_J,
$$
and is completely alternating.
\end{proposition}
\begin{proof}
From Lemma \ref{fGprop} below (with $\nabla$ the Levi-Civita
connection) we have that in any case $G(\cdot,\cdot,\cdot)$ is skew on
the first two arguments.  If $(M,J,g)$ is nearly K\"ahler then this is
also skew on the last pair.  That used
Proposition \ref{cin} and from that Proposition the converse direction
is immediate. 

Since $G(\cdot,\cdot,\cdot)$ is completely alternating and
anti-Hermitian over the first pair (by \nn{Ganti}) it follows that
$G(\cdot,\cdot,\cdot)$ is anti-Hermitian over any pair of
arguments. Thus $G_+=0$. The final statements are then immediate from
Proposition \nn{nj} and the expression \nn{Ttf} for the torsion of $\con{G}$. 
\end{proof}

 Beginning with the Levi-Civita connection $\nabla$ then from the
 family of almost complex connections of Proposition \ref{gencon} it
 is easily verified that $\con{G}$ is the unique connection that
 preserves the metric. This follows by a minor adaption of the proof
 of Theorem \ref{tproc} below.

A powerful feature of the almost Hermitian setting is that the torsion
carries the same information as $G$. This enables us to characterise
the special connection $\con{G}$ in terms of torsion, as follows (and cf.\ \cite{GauHconn,Lib12,L13} ).
\begin{theorem}\label{char}
Let $(M,J,g)$ an almost Hermitian structure. On this there is a  unique
almost complex metric connection with torsion $T$ satisfying the
algebraic condition
\begin{equation}\label{gcond}
G^T_g( X , Y )-JG^T_g(J X, Y ) =0 \quad \forall X,Y\in \Gamma(TM),
\end{equation}
where
\begin{equation}\label{gT}
(G^T_g)^a{}_{bc}: =\frac{1}{2}(T_{c}{}^a{}_{b}-T^a{}_{bc}-T_{bc}{}^a).
\end{equation}
This connection is $\con{G}$, based on the Levi-Civita connection. 
\end{theorem}
\begin{proof} First observe that beginning with the Levi-Civita 
connection and forming from it $\con{G}$ we have that $\con{G}$ is an
almost complex metric connection. 

Now let $\con{T}$ be any connection that is metric, that is
$$
\con{T} g=0 
$$ and write $T$ for its torsion.  Then the difference tensor (as
connections on $TM$) $\con{T}-\con{LC}$ is $G^T_g$ (that is $
G^T_g(X,Y)=\con{T}_XY-\con{LC}_XY$) as given in \nn{gT}.  The
condition \nn{gcond} is the statement that the complex linear part of
$ G^T_g(\cdot,X) $ is zero, for all $X\in \Gamma (TM)$. The $G$ from
Proposition \ref{gener} (with $\nabla$ the Levi-Civita for $g$) has
this property, so $\con{G}$ provides an almost complex metric
connection connection satisfying the conditions \nn{gcond} and
\nn{gT}.

Let us now consider any metric connection $\con{T}$ satisfying
\nn{gcond}.  Then $ G^T_g(\cdot,X) $ is complex anti-linear and so by
the second part of Proposition \ref{gener} (again applied using $\nabla$ set
to be the Levi-Civita connection for $g$) $G^T_g= G$. Thus
$\con{T}=\con{G}$.
\end{proof}

For each structure in the Gray-Hervella classification one might hope
that there is a corresponding {\em characteristic connection}
\cite{Nurowski}. Here this means an almost complex metric connection
with torsion, but with torsion in some sense algebraically minimal so
that with this torsion condition there exists a connection satisfying
the given conditions, and it is unique.  (This generalises the use of
the term ``characteristic connection'' in the works of e.g. Friedrich
\cite{Friedrich}, see also \cite{Agricola} for a review.)

The Theorem \ref{char} (or any of its equivalents in the literature)
provides such a connection.  One simply translates each of the
structures in the Gray-Hervella into a condition on the $G$ formed
from the Levi-Civita connection (as for the examples in Proposition
\ref{cin}). Now one uses the formula \nn{gT} to recast the condition
on $G$ as a restriction on torsion. This combined with \nn{gcond} give
the total conditions to be imposed on the torsion.  The existence and
uniqueness then follow from Theorem \ref{char}. Although this result
is well known we summarise it here for comparison with the conformal and
projective cases below:
\begin{corollary}\label{charGH}
There is a canonical characteristic connection for each of the
structures in the Gray-Hervella classification of almost Hermitian manifolds.
\end{corollary}

\section{Conformal almost Hermitian manifolds}\label{confS}

Throughout this section we take $(M^n,J)$ to be an almost complex
manifold of dimension $n\geq 4$. A (Riemannian) conformal
structure  $c$ on $M$ is an equivalence class of Riemannian metrics
such that if $g,\widehat{g}\in c$ then $\widehat{g}=e^{2\phi}g$ for
some $\phi\in C^\infty (M)$.  Here we observe that the structure $(M^n,J,c)$ 
determines several canonical affine connections with different
characterising properties. Much of the below will work for Hermitian
metrics in signatures $(2p,2q)$, but for simplicity we restrict to the
Riemannian setting.

Note that if $J$ is orthogonal for $g\in c$ then it is orthogonal for
all metrics in $c$. In this case we shall say that $(M,J,c)$ is a {\em
  conformal almost Hermitian} structure.  Note also that, by the
observation of Remark \ref{av}, a Riemannian conformal structure on
$(M,J)$ determines an almost Hermitian Riemannian conformal structure
$c_+$. We shall henceforth assume that any conformal structure $c$ is
almost Hermitian.

\subsection{A canonical torsion free connection} \label{Baileysec}

An affine connection $\nabla$ on a Riemannian manifold
$(M,g)$ will be said to be {\em conformal} if it preserves the
conformal class of the metric, that is
\begin{equation}\label{Weylc}
\nabla_a g_{bc}= 2 B_a g_{bc},
\end{equation}
for some 1-form field $B$ that we shall term the {\em Weyl potential}.
A {\em Weyl connection} $\con{W}$ is an affine connection which is
conformal and torsion free \cite{Weyl}.  

On a conformal structure
$(M,c)$ we shall say an affine connection $\nabla$ is {\em conformal} (or
{\em Weyl} if torsion free) if \nn{Weylc} holds for all $g\in c$. This
means that on a conformal structure $(M,c)$ there is not a  Weyl potential
$B_a$, but  rather an equivalence class of such
over the conformal equivalence relation: given $g\in c$, $B^g_a$ is a
1-form field and if $\widehat{g}=e^{2\phi} g$, for some smooth
function $\phi$, then 
\begin{equation}\label{Bt}
B^{\widehat{g}}_a=B^g_a+\Up_a 
\end{equation}
where $\Up:=d\phi$. \ (In fact $B$ is a connection coefficient, as we
explain in Section \ref{ci} below.)  That such structures arise
naturally is illustrated by the following result of
\cite{Vaisman} (and see also \cite{Bailey}). 
\begin{proposition}\label{Baileyp}
  Let $(M,J,c)$ be a conformal almost Hermitian structure of dimension
  $n\geq 4$. There is a canonical and unique Weyl connection $\con{c}$
  that is compatible with $J$; that is satisfying
$$
\con{c}_aJ^a{}_b=0.
$$ 

Given a choice of $g\in c$, $\con{c}$ is given explicitly in terms of
the Levi-Civita connection $\nabla$ and $J$ by
\begin{equation}\label{Bcon}
\con{c}_a Y^b = \nabla_aY^b -B_a Y^b+ B^b Y_a -B_cY^c \delta^b_a
\end{equation}
where 
\begin{equation}\label{Bform}
B^g_a:= \tfrac{1}{n-2}J^c{}_{b}\nabla_c J^{b}{}_a.
\end{equation}
\end{proposition}
\noindent{\bf Proof:} Let us fix $g\in c$.
The formula \nn{Bcon} for $\con{c}$ is equivalent to the
formula for its dual:
\begin{equation}\label{dual} 
\con{c}_aU_b=\nabla_a U_b + B_a U_b + B_b U_a - U^c B_c g_{ab} ,
\end{equation} 
where $U$ is any 1-form field. From this \nn{Weylc} follows, and
conversely it is easily verified that \nn{Weylc}, with the torsion
free condition, implies \nn{dual}.

Next note that using \nn{Bcon} and \nn{dual} we have
\begin{equation}\label{beep}
\con{c}_i J^k{}_\ell=\nabla_i J^k{}_\ell+ \big(
B^k J_{il}+B_\ell J^k{}_i - B_a J^a{}_\ell \delta^k_i - J^k{}_a B^a g_{\ell i}\big).
\end{equation}
Contracting this yields that 
$$ 
0=\con{c}_i J^i{}_\ell \Leftrightarrow 0= \nabla_i J^i{}_\ell + (B_iJ^i{}_{\ell}-nB_i
J^i{}_\ell +J^i{}_{\ell} B_i),
$$
and thus
$$ 
0=\con{c}_i J^i{}_\ell \Leftrightarrow \nabla_i J^i{}_\ell = (n-2)B_iJ^i{}_\ell \Leftrightarrow B_a= \tfrac{1}{n-2}J^c{}_{b}\nabla_c J^{b}{}_a .
$$ Since $B_a$ is uniquely determined, and $g\in c$ is arbitrary, this proves
the Proposition and that in particular $\con{c}$ depends only on $J$
and $c$ (but not the further information of $g\in c$).  \quad $\Box$

Observe that in the proof above the factor $(n-2)$ arising shows that
the property of compatibility is stable under conformal rescaling. In fact 
in dimension 2 every Weyl connection is complex.

\begin{remark}\label{Leef} Note that it is immediate from the Proposition 
\nn{Baileyp} that $B_a$, as defined in \nn{Bform}, must satisfy the
conformal transformation formula \nn{Bt}; this can also be verified
using the conformal transformation properties of the Levi-Civita
connection (see e.g. \cite{BEG}).  In the literature
(e.g.\ \cite{GrayH}) $2B^g$ is usually called the {\em Lee form},
c.f.\ \cite{Lee}.

Since $\con{c}$ is conformally invariant it follows that \nn{beep}
defines an invariant of $(M,J,c)$; this is precisely
the conformal invariant $\mu$ of \cite[Section 4]{GrayH}.
\end{remark}

With a canonical conformally invariant connection, as we have with
$\con{c}$, many geometric consequences are immediate. For example we
have an immediate consequence of the Proposition \ref{Baileyp}. 
\begin{corollary}\label{cgeod}
Let $(M,J,c)$ be a conformal almost Hermitian manifold of any dimension.
Then $M$ has a preferred class of parametrised curves, viz.\ the geodesics of 
$\con{c}$.
\end{corollary}

Furthermore, since $\con{c}$, from Proposition \ref{Baileyp}, is
canonically determined by a conformal almost Hermitian structure, it
may be used to proliferate (conformally invariant) invariants of the
structure. For example, the curvature $\stack{c}{R}$ of $\con{c}$ is
an invariant of the structure $(M,J,c)$. Its $\con{c}$ covariant
derivatives $\con{c}\cdots \con{c}\stack{c}{R}$, the $\con{c}$
derivatives of $J$ and contractions thereof also yield invariants and
so forth.  However for many purposes it is obviously more natural to
work rather with a connection that is both conformally invariant and
preserves the almost complex structure $J$.

\subsection{Canonical conformal almost complex connections}\label{ccct}
Since the conformal almost Hermitian structure $(M,J,c)$ determines
$\con{c}$ it follows that also canonically associated to the structure
 is the invariant
\begin{equation}\label{Gconf}
\G{c}^b{}_{ca}:= \frac{1}{2}(\con{c}_a J^b{}_d)J^d{}_c~;
\end{equation}
from the earlier developments it is clear this should play a 
fundamental role. 

More generally, by using the canonical torsion-free affine connection
$\con{c}$ as the initial connection, and using the results of Section
\ref{accalc}, we can form a range of geometric objects determined
canonically by the conformal structure and the compatible almost
complex structure $J$.  We begin this with the following.
\begin{proposition}\label{generc} Let $(M,J,c)$ be a conformal almost Hermitian structure of dimension
  $n\geq 4$ and $g\in c$. This structure determines a canonical affine
  connection $\con{gc}$ defined by
$$
\con{gc}_X Y = \con{c}_X Y+ \G{c}(Y,X),
 \quad \mbox{where} \quad
\G{c}^b{}_{ca}:= \frac{1}{2}(\con{c}_a J^b{}_d)J^d{}_c~.
$$ This has the properties:\\ $\bullet$ $\con{gc}_a g_{bc}= 2 B_a
g_{bc}$,\\ 
 $\bullet$ $\con{gc} J=0$;\\ $\bullet$ The
anti-Hermitian part of its torsion gives the Nijenhuis tensor 
$N_J(X,Y)= T^{gc}_{-}(X,Y)$.\\
\end{proposition}
\noindent{\bf Proof:} All results are simply specialisations of
statements in Proposition \ref{gener} and Corollary \ref{torG} except
for the fact $\con{gc}_a g_{bc}= 2 B_a g_{bc}$. This is a consequence
of the corresponding property of $\con{c}$ in Proposition \ref{Baileyp}, and
the first part of the Lemma below.  \quad $\Box$
\begin{lemma} \label{fGprop} On an almost Hermitian manifold $(M,J,g)$
let $\nabla$ be any affine connection with the property that $\nabla_a
g_{bc}= 2 B_a g_{bc}$ for some 1-form $B$ (i.e.\ $\nabla$ is
conformal). Then with $\G{}^a{}_{bc}$ defined as in Proposition
\ref{gener}, $\G{}(X,Y,Z):= g(X,\G{}(Y,Z))$, and $\G{}_{\pm}(X,Y,Z):=
g(X,\G{}_{\pm}(Y,Z))$, we have\\ $\bullet$
$\G{}(X,Y,Z)=-\G{}(Y,X,Z)$;\\ $\bullet$ $\G{}(X,JY,Z)=-
\G{}(Y,JX,Z)$;\\ $\bullet$ $\G{}_+(X,Y,Z)=-\G{}_+(Y,X,Z)$, and
$\G{}_{-}(X,Y,Z)=-\G{}_{-}(Y,X,Z);$\\ $\bullet$ $G(JX,JY, Z)=- G(X,Y,
Z)$, and $G_\pm(JX,JY, Z)=-G_\pm(X,Y, Z)$.
\end{lemma}
\noindent{\bf Proof:} The last claim is immediate from the complex
anti-linearity of $G$, as in Lemma \ref{Gprop}, and the fact that $J$
is orthogonal for $g$.  

It remains to establish the first two claims,
since these imply the third. (The parts $\G{}_+$ and $\G{}_{-}$ are as
defined in Section \ref{accalc}.) For the first we have
$$
\begin{array}{l}
2\big( G(Y,Z,X)  +   G(Z,Y,X)\big)\\ 
\hspace*{10mm} = g(J(\nabla_X J)Y,Z)+g(Y,J(\nabla_X J)Z)\\
\hspace*{10mm} = -g((\nabla_X J)Y,JZ)-g(JY,(\nabla_X J)Z)\\
\hspace*{10mm} = -g(\nabla_X (JY),JZ)-g(JY,\nabla_X (JZ))+g(J \nabla_X Y,JZ)+g(JY,J \nabla_X Z)\\
\hspace*{10mm} = - X\cdot g(JY,JZ) + (\nabla_Xg)(JY,JZ) + X\cdot g(Y,Z)-(\nabla_Xg)(Y,Z)\\
\hspace*{10mm} = 0,
\end{array}
$$
where the last equality follows using again that $g$ is almost Hermitian 
and that $\nabla_Xg=2B(X)g$.

Now for the second identity we calculate 
$$
G(X,JY,Z)+  G(Y,JX,Z)
$$
Setting $X'=-JX$ this is 
$$
\begin{array}{ll}
g(JX',G(JY,Z))-  g(Y,G(X',Z)) & = g(JX',G(JY,Z))-  g(JY,JG(X',Z))\\
& = g(JX',G(JY,Z))+  g(JY,G(JX',Z))\\
&=0,
\end{array}
$$
where in the last line we have used the previous result.
\quad $\Box$

In particular the identities of Lemma \ref{fGprop} hold for $\G{c}$.
From these we obtain the following conformal analogue of Proposition \ref{cin}, part H.
\begin{proposition}\label{ht}
  If $(M,J,c)$ is a conformal Hermitian structure then $\G{c}_-=0$ and
  the connection $\con{gc}$ has Hermitian torsion.
\end{proposition}
\begin{proof} From Proposition \ref{generc} we have $N_J(X,Y)=
  \G{c}_-(Y,X)-\G{c}_-(X,Y)$. So if $N_J=0$ then
  $\G{c}_-(Y,X)=\G{c}_-(X,Y) $. This with the third bullet point of
  Lemma \ref{fGprop} (applied in the case $\nabla=\con{c}$, so $G$
  there is $\G{c}$) implies $\G{c}_-=0$. Thus the torsion is
  Hermitian: $\stack{gc}{T}(X,Y)=\G{c}_+(Y,X)-\G{c}_+(X,Y).$
\end{proof}

\subsection{Compatibility in the conformal setting}
We note here that in this setting there is a refinement
of Lemma \ref{compl} (related to the forming of metric traces).
\begin{lemma}\label{compl2}
  On a conformal almost 
Hermitian structure $(M,J,c)$ a Weyl connection
  $\con{W}$ is compatible, and then agrees with $\con{c}$, if and only
  if $\G{W}$ is totally trace-free; equivalently, if and only if
  $J\G{W}$ is totally trace-free; equivalently if and only if
  $\G{W}(\cdot,J\cdot)$ is completely trace-free; equivalently if and
  only if $\G{W}(J\cdot,J\cdot)$ is completely trace-free.
\end{lemma}
\noindent{\bf Proof:} Let $\con{W}$ be a Weyl connection as in
\nn{Weylc} (and $\con{W}$ is torsion free). Recall from Lemma
\ref{tffs} that $\G{W}^i{}_{ik}=0$, $(J\G{W})^i{}_{ik}=0$, and these
properties are not linked to compatibility.

From Lemma \ref{compl} (and using Lemma \ref{tffs}) we have that
compatibility is equivalent to the vanishing of the trace, for all
tangent fields $X$, of any one of $\G{W}(X,\cdot)$, $J\G{W}(X,\cdot)$, 
$\G{W}(X,J \cdot)$.

Since a conformal structure is available we may also use
a metric $g\in c$ to form a trace:
$$
\begin{array}{ll}
2g^{jk}\G{W}^i{}_{jk}&= g^{jk} (\con{W}_k J^i{}_\ell )J^\ell{}_j= -  g^{jk} J^i{}_\ell \con{W}_k J^\ell{}_j\\
&= -J^i{}_\ell \con{W}_k(g^{jk}J^\ell{}_j)
+J^i{}_\ell J^\ell{}_j \con{W}_k g^{jk}\\
&=  J^i{}_\ell \con{W}_k(g^{\ell m}J^k{}_m)- \con{W}_k g^{ik}\\
&=J^i{}_\ell g^{\ell m} \con{W}_kJ^k{}_m+  J^i{}_\ell J^k{}_m\con{W}_kg^{\ell m}- \con{W}_k g^{ik}\\
&=J^{im} \con{W}_kJ^k{}_m -2B^i+2B^i=J^{im} \con{W}_kJ^k{}_m.
\end{array}
$$ 
This vanishes if and only if $\con{W}$ is compatible. If it is
compatible then by Proposition \ref{Baileyp} $\con{W}=\con{c}$.

On the other hand 
$$
g^{jk}(J\G{W})^i{}_{jk}= J^i{}_\ell  g^{jk}\G{W}^\ell{}_{jk},
$$ 
and so this trace vanishes if and only if $g^{jk}\G{W}^\ell{}_{jk}=0$,
whence the claim for $JG$ follows from the previous.

For the last two parts note that the metric trace of
$\G{W}(J\cdot,J\cdot)$ is obviously the same as the metric trace
of $G$ (since $J$ is orthogonal)
while the metric trace of $2\G{W}(\cdot,J\cdot)$ is 
$-g^{ca}\nabla_bJ^b{}_a $.
 \quad $\Box$

\subsection{Characterising the distinguished connection}\label{ccharsec}

  We may use $\con{gc}$ as the $\con{G}$ in Proposition \ref{gencon} to
  give a 1-parameter family $\con{c,t}$ of connections determined by
  the conformal almost Hermitian structure.
Acting with this on $g\in c$ we have
$$
\con{c,t}_a g_{bc}= 2B_a g_{bc}- t(g_{dc}\G{c}_+^{d}{}_{ab}+g_{bd}\G{c}_+^{d}{}_{ac})
$$ 
so this satisfies a conformal condition $ \con{c,t}_a g_{bc}=
B'_ag_{bc} $ if and only if
$t(\G{c}_{+cab}+\G{c}_{+bac})=2(B_a-B'_a)g_{bc}$. But, from Lemma \ref{compl2},
$\G{c}_+$ is trace-free 
and so either $t=0$ or 
$$
\G{c}_+(X,Y,Z)=-\G{c}_+(Z,Y,X). 
$$ But from Lemma \ref{fGprop} $\G{c}_+$ is also alternating on the
first two arguments and hence if the display holds then altogether we
have $\G{c}_+\in \Lambda^3$. But this implies $\G{c}_+=0$ as by
definition $\G{c}_+$ is Hermitian in the last pair of arguments,
whereas from Lemma \ref{fGprop} it is also anti-Hermitian in the first
pair.  In summary we have the following.
\begin{theorem}\label{tproc}
The connection $\con{c,t}$ defined by
\begin{equation}\label{tclass}
\con{c,t}_X Y: =\con{gc}_X Y+t \G{c}_+(X,Y) 
\end{equation}
is almost complex for any $t\in \mathbb{R}$. It is conformal 
if and only if $t=0$.
\end{theorem}

The Theorem shows that, on any conformal
almost Hermitian structure,  $\con{gc}$ is certainly distinguished among the
natural class $\con{c,t}$.
 We seek a characterisation which is in the
spirit of the characterisation of the Levi-Civita connection on
Riemannian manifolds as the unique torsion free connection preserving
the metric. First a preliminary result that in fact is stronger than we need.

Here we weaken the property of preserving $J$ to just compatibility
and conformality. We show that compatible conformal connections are
parametrised by the algebraic torsion tensors, that is sections of $TM\otimes(\Lambda^2 T^*M)$.
\begin{theorem}\label{torsionB}
  Let $(M,J,g)$ be an almost Hermitian manifold. Let $T$ be a (1,2)
 tensor satisfying  $T^a{}_{bc}=-T^a{}_{cb}$ but which is otherwise 
arbitrary. Then there exists a unique 1-form $B'$ and a unique 
connection $\con{T}$ such that:\\
  $\bullet$ $\con{T} g= 2B' g$;\\
  $\bullet$ $T$ is the torsion of $\con{T}$;\\
 $\bullet$
  $\con{T}_aJ^a{}_b=0$.
\end{theorem}
\noindent{\bf Proof:} It is a straightforward calculation to verify
that the unique solution is given by the following:
$$
B'_d=A_d+B_d 
$$
where
\begin{equation}\label{Apart}
\A{T}_d:= \frac{1}{n-2}\big( T^a{}_{ad}+ \frac{1}{2}J^a{}_cJ^b{}_d
(T_a{}^c{}_b-T_{ba}{}^c- T^c{}_{ba})
\big);
\end{equation}
and the difference tensor $\G{T}$, defined by
$\G{T}(Y,X):=\con{T}_X Y-\con{c}_X Y$, is given by
\begin{equation}\label{Gpart}
\G{T}_{abc}= \A{T}_ag_{bc}-\A{T}_bg_{ac}-\A{T}_cg_{ab}+\tfrac12(T_{cab}-T_{abc}-T_{bca}).
\end{equation}
\quad
$\Box$\\
\noindent 
Note that because $g$ was arbitrary (apart from the condition of being
almost Hermitian) in the calculations of the proof, and
$B'_d=A_d+B_d$, it follows at once that 
 $A_b$ and hence $\G{T}$ (with the latter as a $(1,2)$
tensor) are conformally invariant. This is also clear by inspection of
the formulae \nn{Apart} and \nn{Gpart}.  These objects
are purely dependent on the torsion of $\con{T}$ and the structure $(M,J,c)$.

 The
torsion $\stack{gc}{T}$ of the connection $\con{gc}$ is a conformal
invariant of a conformal almost Hermitian structure,
$$
\stack{gc}{T}(X,Y): =\G{c}(Y,X)-\G{c}(X,Y).
$$ On the other hand, on a fixed conformal almost Hermitian manifold
it follows from Theorem \ref{torsionB} that the conformal almost
complex connections are, in particular, parametrised by their
torsions. Thus we have the following characterisation of $\con{gc}$.
\begin{proposition}\label{tpar}
Let $(M,J,c)$ be a conformal Hermitian structure, then 
$\con{gc}$ is the unique conformal almost complex connection with torsion $\stack{gc}{T}$.
\end{proposition}
 
Theorem \ref{torsionB} and the proof of Proposition \ref{gener}
enable us to invert the last observation and obtain a torsion 
characterisation of the canonical connection $\con{gc}$ of Proposition
\ref{generc}. First we summarise the presence of conformal invariants that we need.
\begin{proposition}\label{confTs} 
On a conformal almost Hermitian manifold $(M,J,c)$ there is a 
conformally invariant map from algebraic torsion tensors to 1-forms given by 
 $$
T^a{}_{bc}\mapsto \A{T}_d ,
$$ where $\A{T}$ is given by \nn{Apart}.  Thus for each algebraic
torsion tensor field we may form the conformally invariant $(1,2)$
tensor
\begin{equation}\label{JTinv}
V(T):= J\G{T}(\cdot,\cdot)+ \G{T}(J\cdot,\cdot)
\end{equation}
where $\G{T}$ is given by \nn{Gpart}.
\end{proposition}

\begin{theorem}\label{cchar}
Let $\nabla$ be an almost complex conformal connection on a conformal almost
Hermitian structure $(M,J,c)$. Then  $ \nabla = \con{gc} $ if and only if
$V(\Tor{\nabla})=0$.
\end{theorem}
\noindent That is, on a conformal almost Hermitian structure,
$\con{gc}$ is the unique conformal almost complex connection with
torsion satisfying the algebraic condition $V(\Tor{\nabla})=0$.  The
Theorem provides a simple torsion condition which characterises this
among all connections preserving the conformal almost Hermitian
structure. Thus it provides the sought conformal almost Hermitian
analogue of the characterisation of the Levi-Civita connection.  As in
the Riemannian case (i.e. as in Proposition \ref{charGH}) we could
impose further conditions on the torsion for each structure in the
Gray-Hervella classification of conformal almost Hermitian structures
and so obtain the following conclusion.
\begin{corollary}\label{charcGH}
There is a canonical characteristic connection for each of the
structures in the Gray-Hervella classification of conformal almost 
Hermitian manifolds.
\end{corollary}

We next look at some of the standard conformal variants of the well
known special almost Hermitian structures. In particular we see how
these fit into the current picture.

\subsection{The Nearly K\"ahler Weyl condition} \label{NKWS}
It is natural to consider the conformal condition 
$$
(\con{c}_X J)X=0 \quad \mbox{for all} \quad  X\in \Gamma(TM);
$$ equivalently   
\begin{equation}\label{NKWcond}
2J\G{c}(X,X)=0 \Leftrightarrow \G{c}(X,X)=0  \quad \forall X\in \Gamma(TM).
\end{equation}
 It is easily verified that this agrees with the
manifold being in the class $\cW_1\oplus \cW_4$ of \cite{GrayH}, it is
the natural conformal analogue of the nearly K\"ahler class, and is
sometimes called {\em nearly K\"ahler Weyl} in the
literature. 
\begin{proposition} \label{nKW}
A conformal almost Hermitian structure $(M,J,c)$ is nearly K\"ahler
Weyl if and only if $\con{c}$ and $\con{gc}$ have the same
geodesics. In particular they are in the same projective class.
\end{proposition}
\begin{proof}
The geodesic equations $\con{c}_X X=0$ and $\con{gc}_XX=0$,
for the two connections,  differ by $\G{c}(X,X)$.
\end{proof}

Since those properties of the Levi-Civita connection and its
associated $G$, that were used to prove Proposition \ref{NKsk}, are
also satisfied by $\con{c}$ and $\G{c}$, we immediately have the
analogous result, as follows. 
\begin{proposition} \label{NKWch}
NKW structures satisfy the following conditions
$$   
\G{c}_+=0 
$$
while
$$ 
\Tor(\con{NKW})=-2 \G{c}_- =N_J,
$$
and is completely alternating.
\end{proposition}

\subsection{The locally conformally almost K\"ahler condition} \label{LCAKS}
Recall that an almost Hermitian manifold $(M,J,g)$ is said to be almost
  K\"ahler if the K\"ahler form $\om = g(\cdot,J \cdot)$ is closed.

On a dimension $n=4$ almost Hermitian manifold the map $\omega\wedge
\cdot:\Lambda^1\to \Lambda^3$ is an isomorphism. Thus one has
\begin{equation}\label{leeeq}
d\omega =\theta\wedge \omega 
\end{equation}
for some 1-form $\theta$. Note that it follows from this that 
$d\theta\wedge \omega=0$. 
If further $d\theta =0$ then we say that the manifold is
{\em locally conformally almost K\"ahler} (LCAK), since locally a
conformally related metric has a closed K\"ahler form.  

For an almost Hermitian manifold of dimension $n\geq 6$ the equation
\nn{leeeq} does not hold generally, but if it does then $d\theta$ is
zero, since then $\omega\wedge \cdot:\Lambda^2\to \Lambda^4$ is
injective. Thus for $n\geq 6$ the structure is LCAK if and only if \nn{leeeq}
holds.

Using the metric $g$ and its Levi-Civita connection $\nabla$,
\nn{leeeq} becomes
$$
\nabla_a \omega_{bc} + \nabla_c \omega_{ab} +\nabla_b \omega_{ca}= 
\theta_a \omega_{bc} + \theta_c \omega_{ab} +\theta_b \omega_{ca} .
$$ 
Contracting  now with $\omega^{bc}$ (indices raised using $g$), and
using $\omega^{bc}\omega_{bc}=n$, gives $(n-2) \theta =2 (n-2)
B $, and so
$$
\theta = 2 B, \quad \mbox{if} \quad  n\geq 4,
$$ 
where the 1-form $B$ is as in \nn{Bform} above. So $\theta$ is the Lee form.

Thus putting together the results for dimensions $n\geq 6$ with also
the definition of LC(A)K in dimension 4, we see that in all dimensions
$n\geq 4$ the invariant 
$$
F:= d B
$$ 
is an obstruction to an (almost)
Hermitian manifold being LC(A)K; as we discuss below $F$ is a
conformal invariant. If $F$ vanishes, and \nn{leeeq} holds, then the
de Rham cohomology class of $[B]\in H^1(M)$ is the obstruction to
$(M,J,g)$ being conformally (almost) K\"ahler.

In dimensions $n\geq 6$ we can capture the full obstruction to the
LCAK condition in terms of $\G{cq}$. Expanding the left and right hand
side of $d\omega=\theta\wedge \om$ in terms of $\con{gc}$ we obtain
$$
\con{gc}_{g[c}\om_{ab]} + \omega_{d[c} \stack{gc}{T}^{d}{}_{ab]} =\theta_{[c}\om_{ab]} ,
$$ where $[\cdots] $ indicates taking the totally skew part and
$\stack{gc}{T}$ is the torsion of $\con{gc}$. But $\con{c}$ preserves
$J$, and acts on $g$ as in Proposition \ref{generc}. So
$\con{gc}_{g[c}\om_{ab]}= 2 B_{[c}\om_{ab]}=
\theta_{[c}\om_{ab]}$. Thus $\omega_{d[c} \stack{gc}{T}^{d}{}_{ab]}=0$
is equivalent to \nn{leeeq}. In summary, and using the expression
\nn{Ttf} for the torsion, we have the following.  Here
$\stack{gc}{T}(\cdot,\cdot,\cdot)$ means
$g(\cdot,\stack{gc}{T}(\cdot,\cdot)) $.
\begin{proposition} \label{LCAK} Let $(M,J,g)$ be an (almost) Hermitian
  manifold of dimension $n\geq 6$.  Then it is LC(A)K if and only if the
  conformal invariant
\begin{equation}\label{first}
\Alt \G{c}(J\cdot,\cdot,\cdot)
\end{equation}
vanishes. 

In the Hermitian case there is also an alternative as follows.  A
Hermitian manifold is LCK if and only if the conformal invariant
\begin{equation}\label{sec}
\Alt \G{c}(\cdot,\cdot,\cdot) \quad \mbox{or equivalently} \quad \Alt \stack{gc}{T}(\cdot,\cdot,\cdot)
\end{equation}
vanishes.
\end{proposition}
\begin{proof} The agreement (up to a constant) of \nn{first} and
  $d\omega-\theta\wedge \om$ was established above (and
  cf.\ e.g.\ \cite{GrayH}).

  From the first part  it follows that $\Alt g ( \cdot,\G{c}(J \cdot,J \cdot))
  $ is an equivalent obstruction. But, if $(M,J,c$) is conformal Hermitian then   $\G{c}$ is
  Hermitian, and so this is $\Alt g ( \cdot,\G{c}( \cdot, \cdot)) $.
\end{proof}

\begin{remark}
  Our condition $\Alt \G{c}(J\cdot,\cdot,\cdot)=0$ is the
  Gray-Hervella conformally invariant condition saying that the
  structure $(M,J,c)$ is in the class $W_2\oplus W_4$. The
  introduction of connection $\con{gc}$ enables us to say that in the
  Hermitian case the obstruction for the LCK is the non-trivial
  presence of a totally skew symmetric part of the torsion of
  $\con{gc}$.
\end{remark}

\subsection{The Gray-Hervella types}\label{cGH}

We summarise here how the conformal Gray-Hervella classification
appears in terms of the conformal intrinsic torsion $\G{c}$. Since the
situation is rather degenerate in dimension 4 we assume here that
$n\geq 6$.

\medskip

\begin{center}
\begin{tabular}{||l|l|l||}
\hline
\hline
Gray-Hervella type& Condition in terms of $\G{c}$& Name\\
\hline
\hline
$W_4$&$\G{c}=0$& LCK\\
\hline
$W_1\oplus W_4$&$\G{c}(X,X)=0$& NKW \\
\hline
$W_2\oplus W_4$& $\Alt \Big(g(\cdot,J\G{c}(\cdot,\cdot)\Big)=0$
& LCAK\\
\hline
$W_3\oplus W_4$&$\G{c}=\G{c}_+$&Hermitian\\
\hline
$W_1\oplus W_2\oplus W_4$&$\G{c}=\G{c}_-$&\\
\hline
$W_1\oplus W_3\oplus W_4$&$\G{c}_-(X,X)=0$&\\
\hline
$W_2\oplus W_3\oplus W_4$&$\Alt \Big(g(\cdot,J\G{c}_-(\cdot,\cdot)\Big)=0$
&\\
\hline
\hline
\end{tabular}
\end{center}

\medskip

\noindent Recall the abbreviations:\\ 
LCAK -- Locally conformally almost K\"ahler; LCK -- Locally
conformally K\"ahler; NKW -- Nearly K\"ahler Weyl.

The $W_i$ are $U(m)$ irreducible parts of the representation $W$
inducing $\Lambda^1\otimes \Lambda^{1,1}$, where $\Lambda^{1,1}$
indicates the Hermitian part of $\Lambda^2$, so $W\cong W_1\oplus
W_2\oplus W_3\oplus W_4 $; for details see \cite{GrayH}. The point is
that (on an almost Hermitian manifold) $\con{LC}\omega$ is a section
of $\Lambda^1\otimes \Lambda^{1,1}$ and the Gray-Hervella type
indicates irreducible parts of $\con{LC}\omega$ that may be not zero
on the structure; more precisely projection to the complementary $W_i$
is certainly 0.  

We could equally use the representation corresponding to $\G{LC}$, and
in that language $W_4$ corresponds to a trace-type part of $\G{LC}$
that we have denoted $B_a$ (cf.\ \nn{beep}). With this understood we
can recover conditions in terms of $\G{c}$ and $B$ for any of the
Gray-Hervella types by combining the various conditions with also the
possibility of setting $B_a$ to zero. The latter eliminates $W_4$, and
breaks conformal invariance. So for example the ``pure'' types arise
from either setting $\G{c}=0$, to obtain the $W_4$-type, or any of the
next three conditions in the table with, in addition, $B_a=0$. At the
other extreme imposing only $B_a=0$ gives the $W_1\oplus W_2\oplus
W_3$ type known as semi-K\"ahler. This enables the possibility of
treating any of the almost Hermitian Gray-Hervella types using the
conformal machinery from this section.

\subsection{Higher conformal invariants} \label{ci}
The primary Gray-Hervella classification is based around the vanishing
of invariants involving just one covariant derivative of $J$. There
are obvious ways to extend this to obtain special structures linked to
higher jets of $J$ and the metric. For example one may look at the
$U(m)$-type decomposition of the Riemann curvature, or alternatively
for conformal questions the $U(m)$-type decomposition of the Weyl
tensor, or the curvature $\stack{gc}{R}$, of $\con{gc}$, or even of
the covariant derivatives of $\stack{gc}{R}$.

Here we wish to describe special conformal conditions which arise from
equations on $B$. These are more subtle because $B$ itself is not
itself conformally invariant (recall \nn{Bt}) and some could be of
interest since they are related to key objects in conformal geometry,
such as the GJMS operators of \cite{GJMS}.
There is also a link with the following question:
\begin{quote}
Consider an even dimensional manifold $M$ that admits an almost
complex structure $J$.  Suppose that $M$ is equipped with a Riemannian
signature conformal structure $c$. Does the structure $(M,J,c)$ have a
distinguished metric $g\in c$?
\end{quote}
 Because $(M,J,c)$ certainly has a distinguished Weyl connection
$\con{c}$, this question is closely related to the properties of $B_a$.
In the case that $M$ is closed (i.e.\ compact without boundary) there
is a celebrated answer due to Gauduchon \cite{Gau84} (see also
\cite{CPessay}): There is a unique (up to homothety) metric $g\in c$
such that $\delta B=-\nabla^a B_a=0$ (i.e.\ $\delta$ is the formal
adjoint of $d$, in the metric scale $g$). 
By construction then this {\em Gauduchon metric} is an invariant of
the structure $(M,J,c)$. We shall show below that, at least in
suitably generic settings, there is another distinguished metric.  To
study this, and other conformal invariants, effectively we need a
small amount of additional background which yields another
interpretation of $B$.

On any smooth $n$-manifold $M$ the highest exterior power of the
tangent bundle $(\Lambda^nTM)$ is a line bundle. Its square
$(\Lambda^nTM)^2$ is orientable; let us assume an orientation.  We
may forget the tensorial structure of $(\Lambda^nTM)^2$ and view this
purely as a line bundle, we shall write $L$ or $E[1]$ for the $2n^{\rm
  th}$ positive root. For $w\in \mathbb{R}$ we denote $L^w$ by $E[w]$.

If $(M,g)$ is a Riemannian manifold then $(\Lambda^nTM)^2$ is
canonically trivialised (and oriented) since $\Lambda^n g^{-1}$ is a
section; for a section $\mu$ of $(\Lambda^nTM)^2 $ the component
function in the trivialisation is obtained (up to a non-zero constant)
by a complete contraction with $\wedge^n g$ (where this notation
means the projection of $\otimes^n g$ onto its part in
$(\Lambda^nT^*M)^2$).

Now consider a conformal structure $(M,c)$. Given any $g\in c$, we
write $\si_g$ for the positive section of $E[1]$ with
$(\si_g)^{2n}=\Lambda^n g^{-1}$, under the identification of $E[2n]$
with $(\Lambda^nTM)^2 $. Evidently on $(M,c)$ there is a canonical
section $\bg$ of $S^2T^* M[2]:=S^2T^* M\otimes E[2]$ (where $S^2$
indicates the symmetric second tensor power) with the property that
given any $g\in c$ we have
$$
\bg=\si^2_g g.
$$ This is called the {\em conformal metric}. Note that any positive
section $\si$ of $E[1]$ determines a metric by $g:=\si^{-2}\bg$, so we
call such a $\si$ a {\em scale}. That $L=E[1]$ is a $2n^{\rm th}$ root
(rather than some other choice) is a convenient for conformal
geometry, thus we often term sections of $E[w]$ {\em conformal
  densities} of weight $w$.

\medskip

In this section we note that there are a number of conformal invariants
associated to almost Hermitian manifolds. We have not tried to be
exhaustive in the treatment here, but rather we have attempted to
indicate some interesting directions.  Suppose then that we have a
conformal almost Hermitian manifold $(M,J,c)$.  First note that, from
the definition of $\si_g$, it follows that for any $g\in c$,
$\con{c}_a\si_g=-B_a \si_g$, and similar for $\con{gc}$, and so we
have the following.
\begin{proposition} \label{cm}
 The canonical connections $\con{c}$, and $\con{gc}$ preserve the
 conformal metric
$$
\con{c}\bg=0  \quad \mbox{and} \quad \con{gc}\bg=0.
$$
\end{proposition}

Recall that each metric $g\in c$ determines, through the corresponding
Levi--Civita connection $\nabla$, a Weyl potential
$$
B^g_a= \tfrac{1}{n-2}J^c{}_{b}\nabla_c J^{b}{}_a.
$$
From this formula, it is easily verified that it transforms
conformally according to \nn{Bt}, i.e.\
$$
B^{\widehat{g}}_a= B^{g}_a + \Upsilon_a,
$$
where $\widehat{g}=e^{2\phi}g$, and $\Upsilon
=d\phi$, and $\phi\in C^\infty(M)$. It follows that if $ B^{g}_a $ is
exact then so is $B^{\widehat{g}}_a$, in which case we shall say that
$(M,J,c)$ is {\em conformally semi-K\"ahler} (CSK); on a conformally
semi-K\"ahler manifold there is a distinguished metric $g\in c$,
namely the metric satisfying
$$ 
B^{g}_a=0 \quad \mbox{equivalently} \quad \delta \om =0 \quad
\mbox{equivalently} \quad \con{c} =\con{LC}^g.
$$

It also follows from \nn{Bt} that the {\em Faraday tensor} 
$$
F=d B
$$ 
is a conformal invariant of $(M,J,c)$. In fact this is the
curvature of $\con{c}$ as a connection on $L^{-1}=E[-1]$, from which
its invariance is immediate.  If $F=0$ for a conformal almost
Hermitian structure $(M,J,c)$ then we are locally in the situation
above, so the structure is {\em locally conformally semi-K\"ahler}
(LCSK).   In an obvious way there are weakenings
of this condition by decomposing $F$ into its Hermitian and
anti-Hermitian parts $F_{\pm}$ (which are obviously conformal
invariants). We may specify that $F_+=0$ or $F_-=0$.

There are also routes to more subtle conformal invariants, and we wish
indicate these.  In the discussion below, $\Lambda^k$ denotes the
bundle of $k$-forms or, by way of notational abuse, the sections of
this bundle; and $\Lambda_k$ is the tensor product of $\Lambda^k$ with
the conformal $(2k-n)$-densities.

The manifold $M$ is oriented by the almost complex structure. 
On oriented dimension 4 manifolds
the bundle map known as the Hodge-star operator is conformally
invariant on 2-forms
$$
\star:\Lambda^2\to\Lambda^2
$$ 
and this squares to 1. The 2-forms decompose
orthogonally into the eigenspaces of this and applying the respective
projections to $F$ we obtain, respectively, the
self-dual/anti-self-dual parts $F^{*\pm}$ of $F$ as conformal
invariants of $(M,J,c)$.  Since almost complex manifolds are naturally oriented we
have these invariants on any 4-dimensional almost Hermitian manifold,
and the vanishing of just one of these gives an obvious weakening of
the locally conformally semi-K\"ahler condition. We may say that
$(M,J,c)$ is {\em Faraday (anti-)self-dual} if $F^{*-}=0$
(respectively $F^{*+}=0$); if one or the other condition holds we may
say the structure is {\em Faraday half-flat}. In an obvious way we may
seek a finer grading by further decomposing the $F^{*\pm}$ into
Hermitian and anti-Hermitian parts $F^{*\pm}_{\pm}$. This is
partially successful. A straightforward calculation shows that:
$$
\rank{(\Lambda^2)^{*+}_+}=1;\quad \rank{(\Lambda^2)^{*+}_-}=2; \quad \rank{(\Lambda^2)^{*-}_+}=3 ; \quad \rank{(\Lambda^2)^{*-}_-}=0.
$$ More precisely $(\Lambda^2)^{*+}_+=\mathbb{R}\om $,
$(\Lambda^2)^{*+}_-=\om^\perp\cap (\Lambda^2)^{*+}$, $(\Lambda^2)^{*-}_+= (\Lambda^2)^{*-}
$, $(\Lambda^2)^{*-}_- =0$, meaning the zero bundle.

The formal adjoint of $d:\Lambda^1\to \Lambda^2$, which we denote
$\delta: \Lambda_2\to \Lambda_1$, is also conformally invariant.  Thus
in dimension 4 the {\em Maxwell current} $\delta F= \delta d B$ is
also conformally invariant, and if this is zero then we may say the
conformal almost Hermitian structure is simply {\em Maxwell}.

In fact there is analogue for all dimensions $n\geq 4$ of the last
result.  In order to state this, and for the subsequent developments,
we shall need the following result from \cite{BrGodeRham}.
\begin{theorem}\label{BGres}
On Riemannian manifolds of even dimension $n\geq 4$ there  are
  natural formally self-adjoint differential operators
$$
Q^g_k: \Lambda^k\to
\Lambda_k \quad k=0,1,\cdots ,n/2+1,
$$
with $Q^g_{n/2}=1$, $Q^g_{n/2+1}=0$ and otherwise with properties as
follows. Up to a non-zero constant scale,
$Q^g_k$ has the form
$$
(d\d)^{n/2-k}+\LOT
$$
and $Q_01$ is the \IT{Branson}
$Q$-curvature. Then
\begin{itemize}
\item Upon restriction to the closed $k$-forms $\cC^k$, $Q_k^{g}$
has the conformal transformation law
\begin{equation}\label{Qoptrans}
Q_k^{\wh{g}}u=Q_k^g u +  \d Q_{k+1}^g d (\phi u)
\end{equation}
where $ \wh{g}=e^{2\phi}g$ with $ \phi$ a smooth function.
\item It follows that the differential operator
$\d Q_{k+1}^g d= :L_k:\Lambda^k\to\Lambda_k$ is conformally invariant.

\item The operators $G_k:\Lambda^k\to \Lambda_{k-1}$ defined by the
  composition $G_k:=\d Q_k$ are conformally invariant on the null
  space of $L_k$. In particular, further restricting, the map
  $G_k:\cC^k\to \Lambda_{k-1}$ is conformally invariant.

\end{itemize}
\end{theorem}

Thus we have the following generalisation of the Maxwell current.
\begin{theorem}\label{G2} 
The (conformally weighted) 1-form field $G_2F=L_1 B$ in $\Lambda_1$ is a local
(conformal) invariant of conformal almost Hermitian manifolds.
\end{theorem}
\begin{remark}
Thus $G_2F=0$ gives a weakening of the LCSK condition. 

  Note that $G_2F$ may be viewed as an analogue (for $\con{c}$ viewed as a
  connection on $L^{-1}$) of the Fefferman-Graham obstruction tensor
  of \cite{FGast}, since the latter can be seen to arise from a
  non-linear analogue of $G_2$ applied to a certain curvature (the
  curvature of the conformal tractor connection) \cite{GoPetObstrn}.
  \end{remark}

Next if the invariant $G_2F=0$ (which holds trivially if $(M,J,c)$ is
LCSK), then we are able to use the third bullet point of Theorem
\ref{BGres}. However some care is required as $B$ is not
invariant. For $g\in c$, in this case we obtain that
$$
Q^g_J:=G_1 B^g
$$
is a $-n$-density that transforms conformally like the $Q$-curvature:
if $\widehat{g}=e^{2\phi}g$ as above then
\begin{equation}\label{QJ}
Q^{\widehat{g}}_J= Q^{g}_J+L_0 \phi,
\end{equation}
where we also note that $L_0$, from Theorem \ref{BGres}, is the
dimension order conformally invariant Laplacian power operator of
\cite{GJMS} (the so-called critical GJMS operator). 
Clearly locally
constant functions are in the kernel of $L_0$, if these give the
entire null space $\cN(L_0)$ then $L_0$ is said to have {\em trivial
  kernel}.  Since $L_0$ is formally self-adjoint \cite{GrZ}, from standard
Fredholm theory we have the following result.
\begin{theorem}\label{specmet} Suppose $(M,J,c)$ is a closed conformal 
almost Hermitian manifold satisfying the conformal condition $G_2F=0$,
and such that the GJMS operator $L_0$ has trivial kernel. There is a unique
preferred unit volume metric $g\in c$ satisfying 
$$
Q^g_J=0.
$$
In the special case that $(M,J,c)$ is CSK, $g$ is the unique (unit volume) 
semi-K\"ahler metric in the conformal class. 
\end{theorem}

On a generic (Riemannian signature) conformal manifold the critical
GJMS operator has trivial kernel, but there are examples where the
kernel is non-trivial (see e.g.\ \cite{ES-Frol,GoJ} and similar
examples are easily constructed). Thus we note the following partial
generalisation of the previous Theorem.
\begin{theorem}\label{int} Suppose $(M,J,c)$ is a closed conformal 
almost Hermitian manifold satisfying the conformal condition $G_2F=0$. Then
$$
I_\phi=\int_M \phi Q_J
$$
is a well-defined (i.e.\ conformally invariant) invariant
for any function $\phi$ in the kernel of $L_0$.

These conformal invariants obstruct the conformal prescription of zero
$Q_J$. That is there is metric $g\in c$ satisfying
$$
Q^g_J=0
$$ if and only if $I_\phi=0$ for all $\phi\in \cN(L_0)$. Such a
metric, if it exists, is unique up to $g\mapsto e^{2\phi}g$ where
$\phi\in \cN (L_0)$.
\end{theorem}
\begin{proof}
  By construction the operator $L_0$ is formally self-adjoint, since
  the $Q^g_k$ are, so the first result follows from \nn{QJ} and
  the conformal weight of $Q_J$. 

The second part is again immediate from standard spectral theory.
\end{proof}

\begin{remark}
  As mentioned any (locally) constant function is in the kernel of
  $L_0$, but $Q_J$ is a divergence so the $I_\phi$ defined in Theorem
  \ref{int} may only possibly be non-trivial on structures where the
  kernel of $L_0$ includes locally non-constant functions.  The above
  Theorems are analogues of results for Branson's Q-curvature and its
  prescription in the cases where the conformal invariant $\int_MQ$ is
  zero. See, \cite{MalchSIGMA} and Proposition
  3.5 of \cite{Goforbid}, which use  \cite{BGaim,BGact}.

Note that in the language of physics, the role of $G_1$ in Theorem \ref{specmet}
and Theorem \ref{int} is as a ``gauge fixing'' operator. Indeed the
$(L_k,G_k)$ form graded injectively elliptic systems and in dimension
four $G_1$ is the Eastwood-Singer conformal gauge fixing operator of
\cite{ESgauge}.
\end{remark}

There are further global conformal invariants available as
follows. From the third bullet point of Theorem \ref{BGres} we have
that $G_1$ is conformally invariant on closed $1$-forms. The
conformally invariant subspace $\cH^1$ of $\cC^1$ consisting of those
closed 1-forms that are also annihilated by $G_1$ is termed the space
of {\em conformal harmonics} (of degree 1) and has dimension at least
as large as the first Betti number, see \cite{BrGoderham}. The
following is an easy consequence of the properties of $Q_1$ and the
conformal transformation formula for $B$.
\begin{theorem}\label{int2} On a closed conformal almost Hermitian manifold $(M,J,c)$, 
satisfying the conformal condition $G_2F=0$,
$$
I_u:=\int_M (B, Q_1 u)   
$$
is conformally invariant, for any 1-form $u\in \cH^1$.
\end{theorem}
\begin{remark}
The invariants $I_u$ here generalise the $I_\phi$ of Theorem \ref{int}. 
If $u=d\phi$ for some function $\phi$, then the condition $u\in
\cH^1$ is equivalent to $L_0\phi=0$, and integrating by parts we obtain
$I_{u}=I_\phi$.
\end{remark}

Finally note that in view of the transformation law \nn{QJ} we can
obviously add to the Branson Q-curvature $Q$ a multiple of $Q_J$ so as
to obtain a local conformally invariant $(-n)$-density.

\subsection{Summary}
Any  conformal almost Hermitian structure $(M,J,c)$ determines a 
 canonical Weyl structure $\con{c}$ with conformally invariant curvature $F$ (viewing $\con{c}$  as a connection on $E[-1]$). There is a natural
hierarchy of curvature conditions that one can consider: 
\begin{itemize}
\item $\con{c}$, and therefore also $F$, general;
\item \begin{enumerate}
\item $G_2 F=0$; or
\item $F_+=0$ or $F_-=0$; and in dimension 4 $F^{*+}_{+}=0$, or
  $F^{*-}_{+}=0$, and/or, $F^{*+}_{-}=0$ (and note that in dimension 4
  if either $F^{*+}=0$ or $F^{*-}=0$ then $G_2F=\delta F=0$);
\end{enumerate}
\item $F=0$, this is LCSK; then
\begin{enumerate}
\item $0\neq [B^g]\in H^1(M)$, where $g\in c$; or  
\item $B^g$ exact, so $\con{c}$ is a Levi-Civita connection, for some
  metric in $c$.
\end{enumerate} 
\end{itemize}
The last of these is the conformally semi-K\"ahler (co-symplectic)
condition; in case that is satisfied then $\con{c}$ is the
unique Levi-Civita connection in the conformal class, that is
compatible with $J$.

\begin{remark}\label{moreg}
None of the results developed in Section \ref{ci} above depend on
$J$ being orthogonal for the conformal structure. The discussion has
assumed this only to link with conditions that are familiar in the
literature (such as LCSK).

Moreover only the results using/claiming ellipticity rely on
Riemannian signature.
\end{remark}

\section{Projective Geometry}\label{PGS}
A projective (differential) geometry consists of a manifold $M$
equipped with an equivalence class $p$ of torsion free affine
connections (we write $(M,p)$); the class is characterised by the fact
that two connections $\nabla$ and $\widehat\nabla$ in $p$ have the
same {\em path structure}, that is the same geodesics up to
parametrisation.  In the following to avoid difficult language it will
be useful to understand $p$ in a slightly different way.  We shall
also use $p$ to mean {\em that set of parametrised curves} such that
each curve in $p$ is a geodesic for some $\nabla\in p$. We shall say
that a connection (in general with torsion) is {\em compatible with
  the path structure} if its geodesics agree with the geodesics of
some $\nabla\in p$.

Explicitly, the connections $\nabla$ and $\widehat\nabla$ are related
by the equation
\begin{equation} \label{ptrans}
 \widehat\nabla_a Y^b = \nabla_a Y^b +\Upsilon_a Y^b + \Upsilon_c Y^c \delta^b_a,
\end{equation}
where $\Upsilon$
is some smooth section of $ T^*M$. Equivalently two connections (on  $ T^*M$)
in the same projective class are related by
\begin{equation} \label{ptrans2}
 \widehat\nabla_a u_b = \nabla_a u_b - \Upsilon_a u_b -\Upsilon_b u_a,
\end{equation}
on any  $1$-form field $u$.

Here we consider almost complex manifolds $(M,J)$ of dimension $n=2m$
($m\in \{ 1,2,\cdots \}$) equipped with a projective structure $p$. 
Associated to any $\nabla\in p$ we may form 
\begin{equation}\label{A}
A^{\nabla}_c:= -\frac{1}{n}(\nabla_a J^a{}_b)J^b{}_c =\frac{1}{n}J^a{}_b(\nabla_a J^b{}_c).
\end{equation}
It is then easily verified that if we change $\nabla $ to
$\widehat\nabla$, as in \nn{ptrans}, then
\begin{equation}\label{At}
A^{\widehat{\nabla}}_c= A^{\nabla}_c +\Upsilon_c.
\end{equation}
We have the following result, which is an analogue of Proposition
\ref{Baileyp}.
\begin{proposition}\label{projJ}
  Let $(M,J)$ be an almost complex manifold of any (even) dimension. If
  this is also equipped with a projective structure $p$ then there is
  a unique connection $\con{p}\in p$ 
which is compatible with $J$.
\end{proposition}
\begin{proof}
  Let $\nabla$ and $\con{p}$ be connections in $p$. Then it
  is a straightforward exercise to show that $\con{p}$ is compatible
  with $J$ (as defined in Section \ref{comS}) if and only if (as a
  connection on $T^*M$) we have
\begin{equation}\label{Jcon}
\con{p}_a  u_b = \nabla_a u_b + A^{\nabla}_a u_b + A^\nabla_b u_a 
\end{equation}
where $A^{\nabla}_a$ is given in terms of $\nabla$ by the formula \nn{A}
Showing this uses that $J$ is trace-
free (since $J^2=-\id$).\end{proof}

Note that from the uniqueness it follows that $\con{p}$ is independent
of $\nabla\in p$, as used in the proof. Alternatively one can see this
by putting together \nn{ptrans}, \nn{ptrans2}, and \nn{At}.

\begin{remark} Putting things in another order we may state things
  (somewhat informally) as follows.  Given an even dimensional
  projective manifold $(M,p)$, any almost complex structure $J$ on $M$
  yields a canonical breaking of the projective symmetry, via
  $\con{p}$.  
\end{remark}

There is an obvious consequence of the Proposition. The set of
curves $p$ has a distinguished subset, as follows.
\begin{corollary}\label{pgeod}
Let $(M,J,p)$ be a projective almost complex manifold of any dimension.
Then $M$ has a preferred class of parametrised curves, namely the geodesics of 
$\con{p}$.
\end{corollary}

\subsection{Canonical projective almost complex connections}\label{cpct}

The curvature of $\con{p}$, and its $\con{p}$ covariant derivatives, are
all invariants of the structure $(M,J,p)$.  
Another invariant is 
\begin{equation}\label{Gproj}
\G{p}^b{}_{ca}:= \frac{1}{2}(\con{p}_a J^b{}_d)J^d{}_c~;
\end{equation}
in analogy with the conformal case, this, and its tensor parts, play a
fundamental role.

Specialising the development and results of section \ref{accalc} to
the case that $\con{p}$ is the initial affine connection, we can form
a range of connections and geometric objects that are determined
canonically by the almost complex structure $J$ and the projective
structure.  In particular we obtain the following result.
\begin{proposition}\label{generp} Let $(M,J,p)$ be a 
  projective almost complex structure. This determines a
  canonical affine connection $\con{gp}$ defined by
$$
\con{gp}_X Y = \con{p}_X Y+ \G{p}(Y,X),
 \quad \mbox{where} \quad
\G{p}^b{}_{ca}:= \frac{1}{2}(\con{p}_a J^b{}_d)J^d{}_c~.
$$ This has the properties:\\ 
$\bullet$ $\con{gp} J=0$;\\ $\bullet$ The
anti-Hermitian part of its torsion gives the Nijenhuis tensor 
$N_J(X,Y)= T^{gp}_{-}(X,Y)$.\\
\end{proposition}
\noindent{\bf Proof:} All results are simply specialisations of
statements in Proposition \ref{gener} and Corollary \ref{torG}.
 \quad $\Box$

 Now we may use $\con{gp}$ as the $\con{G}$ in Proposition
 \ref{gencon} to give a family of connections $\con{p,t}$ determined
 by the projective almost complex structure, and parametrised by $t\in
 \mathbb{R}$. Recall that in the conformal setting the canonical
 connection $\con{gc}$ had the congenial property that it preserved
 both $J$ and the conformal structure. The projective analogue would
 be a connection that preserves both $J$ and the path structure of the
 projective geometry. Indeed since, even more, there are preferred
 geodesics (by Corollary \ref{pgeod}) it is reasonable to seek a
 connection that shares this preference. 
 We investigate using the family $\con{p,t}$.

For each choice of $t$, the connection $\con{p,t}$ has the same
 geodesics as the connection $\con{p}$ (and so the same paths as any
 connection in $p$) if and only if 
$$
\G{p}(X,X)+ t\G{p}_+(X,X)=0 \quad \mbox{for all} 
\quad X\in \Gamma(TM) .
$$
This must also hold if $X$ is replaced by $JX$. Together we obtain an
equivalent system 
$$
\begin{array}{c}
(1+t)\G{p}_+(X,X)+ \G{p}_-(X,X)=0\\
(1+t)\G{p}_+(X,X)- \G{p}_-(X,X)=0.
\end{array}
$$
In solving this there are two cases:
if $t\neq -1$ then $\G{p}_+(X,X)= \G{p}_-(X,X)=0$ whence $\G{p}(X,X)=0$; if $t=-1$ then 
we have the weaker requirement $\G{p}_-(X,X)=0$.
 Since the condition $2J\G{p}(X,X)=(\con{p}_X J)X=0$, for all $X\in
\Gamma(TM)$, is an analogue of the nearly K\"ahler condition (on
almost Hermitian structures), if this holds we shall say the structure
$(M,J,p)$ is {\em projective nearly K\"ahler}.  We summarise as
follows.
\begin{theorem}\label{tprop}
The connection $\con{p,t}$ defined by
$$
\con{p,t}_X Y: =\con{gp}_X Y+t \G{p}_+(X,Y) 
$$
is almost complex. It has the same path structure as 
$(M,J,p)$ if and only if:
\begin{itemize}
\item $t\neq -1$ and the manifold is projective nearly K\"ahler, i.e.\
$$
(\con{p}_X J)X=0 \quad \mbox{ for all } \quad X\in \Gamma(TM);
$$ 
\item $t= -1$ and the manifold satisfies 
\begin{equation}\label{pcond}
(\con{p}_X J)X=(\con{p}_{JX}J)(JX)  \quad \mbox{ for all } \quad X\in \Gamma(TM).
\end{equation}
\end{itemize}
\end{theorem}

  It is evident from the Theorem that the connection 
$$
\con{JP}:=\con{p,-1}
$$ is distinguished, since it imposes the weakest condition on the
structure $(M,J,p)$ in order to have the geodesics from the
distinguished set of Corollary \ref{pgeod}. Since the value of any of
these connections is limited without path structure preservation we
make the following definition.
\begin{definition}\label{PJcomp}
On an even manifold $M$ let $J$ be an almost complex structure and $p$
a projective structure. We shall say that $p$ and $J$ are {\em compatible} if 
$$
\G{p}_-(X,X)=0, \quad \forall X ,
$$ or equivalently if \nn{pcond} holds. When this holds we shall also say
that $(M,J,p)$ is a compatible projective almost complex structure.
\end{definition}

Then for emphasis we may state the following. 
\begin{corollary}\label{punch}
On any compatible projective almost complex structure $(M,J,p)$ the
canonical connection $\con{JP}$ preserves $J$ and has as geodesics the
distinguished curves determined by $p$ and $J$ (as in Corollary
\ref{pgeod}).
\end{corollary}
\begin{remark}
The term ``compatible'' as defined here is not easily confused with
its use to describe properties of a particular affine connection, as
defined in Section \ref{comS}.

  Together equation \nn{pcond} and the compatibility equation
  $\con{p}_a J^a{}_b=0$ are {\em formal} analogues of the equations
  used by Gray and Hervella \cite{GrayH} to characterise the class
  $W_1\oplus W_3$ in almost Hermitian geometry. It is more closely an
  analogue of the conformal class $W_1\oplus W_3\oplus W_4 $.

  The $t=1$ case of Theorem \ref{tprop} recovers the projectively
  canonical version of the $\con{KN}$ connection which has torsion
  precisely agreeing with the Nijenhuis tensor.
\end{remark}

\begin{lemma}\label{diffT}
Two affine connections $\nabla$ and $\nabla'$ have the same geodesics
if and only if their difference tensor is an algebraic torsion tensor.
\end{lemma}
\begin{proof}
Let $T$ be a section of $TM\otimes (\Lambda^2 T^*M)$. Then it is clear
that $\nabla'$ defined by $\nabla'_XY=\nabla_XY+\frac{1}{2}T(X,Y)$ has
the same geodesics as $\nabla$.

Conversely suppose that $\frac{1}{2}T$ is the difference tensor given by
$\nabla'-\nabla$ and $T$ is not a section of $TM\otimes (\Lambda^2
T^*M)$. In this case there exists a point $q\in M$ and $X_q\in T_qM$ such that
at $q$ we have $T(X_q,X_q)\neq 0$. Then the $\nabla$-geodesic through
$q$ that has tangent there $X_q$ is not a geodesic for $\nabla'$.
\end{proof}

In the following we shall use the following notation. Given a $(1,2)$ tensor
$H$ we will denote its symmetric  and skew parts by $H^{\rm symm}$ and $H^{\rm skew} $, respectively. That is
$$
H^{\rm symm}(X,Y)=\frac{1}{2}(H(X,Y)+H(Y,X)), 
\quad H^{\rm skew}(X,Y)=\frac{1}{2}(H(X,Y)-H(Y,X)).
$$

The interpretation of the solution leading to Theorem \ref{tprop} is
clear via the Lemma \ref{diffT}. The difference
$(\con{p,t}-\con{p})(X,Y)$ is $\G{p}(Y,X)+t\G{p}_+(X,Y)$. So for
example at $t=-1$ this is $\G{p}^{\rm skew}_+(Y,X)+
\G{p}_-(Y,X)$. Thus the $J,p$ compatibility condition $\G{p}_-^{\rm symm}=0$ is
exactly what is required to achieve a skew difference tensor, as
required since $\con{p}$ torsion free.

At this point it is reasonable to ask if there is an almost complex
connection with the same geodesics as $\con{p}$, without imposing the
$J$, $p$ compatibility condition $\G{p}_-^{\rm symm}=0$. In other
words whether might be some connection that is ``better'' than
$\con{JP}$ in this sense. 
We shall now  show that there is no such connection. 
\begin{theorem}\label{projR}
Let $(M,J,p)$ be a projective almost complex structure. Suppose that
there exists an affine connection $\nabla'$ such that $\nabla'$ is
almost complex and $\nabla'$ has geodesics in the projective class
$p$.  Then $J$ and $p$ are compatible, that is
$$
\G{p}^{\rm symm}_-=0,
$$
and 
$$
\nabla'=\con{JP}+\frac{1}{2}T
$$
where $T$ is an anti-Hermitian algebraic torsion tensor (that is 
$T(X,Y)=-T(Y,X)$ and $T(JX,JY)=-T(X,Y)$) and $\nabla'$ has the same geodesics as $\con{p}$.
\end{theorem}
\begin{proof}
Let us assume that $\nabla'$ preserves $J$ and actually has the same
geodesics as $\con{p}$. Then since $\nabla'$ preserves $J$, and using Proposition
\ref{gener}, we have that
$$
\nabla'_XY=\con{gp}_XY+K(Y,X)
$$
where $K$ is a $(1,2)$ tensor which is complex linear in the first argument.
Now using that  $\nabla'$ has the same geodesics as $\con{p}$
then, using the Lemma \ref{diffT}, this implies that 
$$
\G{p}(Y,X)+K(Y,X) =\frac{1}{2}T(Y,X)
$$ 
for some $(1,2)$ tensor $T$ that is skew, i.e.,
$T(X,Y)=-T(Y,X)$. Taking the symmetric part of both sides gives
\begin{equation}\label{key1}
\G{p}^{\rm symm}+K^{\rm symm}=0.
\end{equation}
Further taking the anti-Hermitian part we have 
\begin{equation}\label{key2}
\G{p}_-^{\rm
  symm}+K^{\rm symm}_-=0 .
\end{equation}
Now from \nn{key1} we have 
$$
\G{p}(JX,X)+\G{p}(X,JX) +K(JX,X)+K (X,JX)=0.
$$
Multiplying through with $J$ gives
$$
\G{p}(X,X)-\G{p}(JX,JX) - K(X,X)+K (JX,JX)=0,
$$
where we have used $\G{p}(\cdot,X)$ is complex anti-linear, while $K(\cdot,X)$
is complex linear.
Thus 
$$
\G{p}^{\rm symm}_-- K^{\rm symm}_-=0,
$$ which with \nn{key2} implies that $\G{p}^{\rm symm}_-=0$, as
claimed.  Since $(M,J,p)$ is thus seen to be a compatible projective
almost complex structure, $\con{JP}$ is a connection with the same
geodesics as $\con{p}$. Thus $\nabla'=\con{JP}+\frac{1}{2}T$ for some
algebraic torsion tensor field $T$. Using now that both $\nabla'$ and 
$\con{JP}$ preserve $J$ it follows immediately that $T$ is anti-Hermitian.

In the above we assumed that $\nabla'$ has the same geodesics as
$\con{p}$. Let us now relax that and assume only that $\nabla'$ (is an
almost complex connection that) has geodesics in the set $p$ given by
the projective structure. Then
$\widehat{\nabla}:=\nabla'-\frac{1}{2}\Tor \nabla'$ is a torsion free
connection with the same geodesics as $\nabla'$ and so
$$
\widehat\nabla_a Y^b = \con{p}_a Y^b +\Upsilon_a Y^b + \Upsilon_c Y^c \delta^b_a,
$$
for some 1-form field $\Upsilon$. Adding ($\frac{1}{2}$ of) the torsion of 
$\con{JP}$, it now follows that 
$$
\con{JP}_a Y^b +\Upsilon_a Y^b + \Upsilon_c Y^c \delta^b_a
$$ 
has the same geodesics as $\nabla'$. The difference between this
connection and $\nabla'$ is some algebraic torsion tensor field,
while both $\con{JP}$ and $\nabla'$ preserve $J$. Thus it is the case
that there is an algebraic torsion tensor field $T$ so that for any fixed $Y$
$$
X\mapsto X\cdot \Upsilon (Y)+ Y\cdot \Upsilon(X) + T(X,Y)
$$ is $J$-linear.  In particular then taking $Y=JX$ and applying this
map to $JX$, the $J$-linearity implies that
$$
2(JX)\cdot \Upsilon (JX)=J(X\cdot \Upsilon (JX))+ J((JX)\cdot \Upsilon (X) )
$$
and so 
$$
(JX)\cdot \Upsilon (JX)=-X\cdot \Upsilon (X), \quad \forall X.
$$
This implies $\Upsilon=0$. Hence $\nabla'$ has the same geodesics as $\con{p}$ and we are reduced to the situation treated first. 
 \end{proof}

Finally, using the notation of the Lemma \ref{diffT}, observe that if
$\nabla'_XY=\nabla_XY+\frac{1}{2}T(X,Y)$ then $T=\Tor
\nabla'-\Tor\nabla$.  Thus the space of connections with the same
geodesics as $\con{p}$ is an affine space through $\con{p}$ and
modelled on the vector space of algebraic torsion tensor fields. In
particular the connections in the class are uniquely parametrised by
their torsion. Thus we have the following result.
\begin{corollary}\label{torch}
On a compatible projective almost complex structure 
$\con{JP}$ is the unique almost complex connection with the same geodesics
as $\con{p}$, and torsion 
$$
-\frac{1}{2}\G{p}^{\rm skew}.
$$
\end{corollary}

\subsection{Projective analogues of the conformal Gray-Hervella 
types}\label{pGH}

Note that $\G{p}$ is trace-free, but in an obvious way this admits an
$SL(m)$ type decomposition into the parts 
$$ \G{p}_+^{\rm symm}, \quad\G{p}_+^{\rm skew},\quad \G{p}_-^{\rm
  symm}, \quad\mbox{and}\quad \G{p}_-^{\rm skew}
$$ 
where for example $\G{p}_+^{\rm
  skew}(X,Y)=\frac{1}{2}(\G{p}_+(X,Y)-\G{p}_+(Y,X))$. It is easily
verified that in dimensions greater than 2 these four components are
functionally independent (see Section \ref{EXS}).

We summarise here how some of the available projectively invariant
conditions we can impose, in terms of this decomposition, correspond
to analogous conditions in the conformal Gray-Hervella classification.

\medskip

\begin{center}
\begin{tabular}{||l|l|l||}
\hline
\hline
Conformal Gray-Hervella type & Usual name & Analogous projective condition\\
\hline
\hline
$\G{c}=0$ & LCK & $\G{p}=0$ \\
\hline
 $\G{c}(X,X)=0$ & NKW & $\G{p}(X,X)=0$ \\
\hline
$\Alt \Big(g(\cdot,J\G{c}(\cdot,\cdot)\Big)=0 $ & LCAK
& none available \\
\hline
$\G{c}=\G{c}_+$ & Hermitian & $\G{p}=\G{p}_+$ \\
\hline
$\G{c}=\G{c}_-$ & &$\G{p}=\G{p}_-$\\
\hline
$\G{c}_-(X,X)=0$ & &$\G{p}_-(X,X)=0$\\
\hline
satisfied identically & &$\G{p}_+(X,X)=0$ \\
\hline
$\Alt \Big(g(\cdot,J\G{c}_{-} (\cdot,\cdot)\Big)=0$ &
& none available \\
\hline
\hline
\end{tabular}
\end{center}

\medskip

\noindent Recall the abbreviations:\\ CNK -- Conformally nearly
K\"ahler; LCAK -- Locally conformally almost K\"ahler; LCK -- Locally
conformally K\"ahler; NKW -- Near K\"ahler Weyl.

Note that the condition $\G{p}_-^{\rm skew}=0$ is the condition for a
projective almost complex structure to be integrable, that is $N_J=0$.
Thus this is implied by (but weaker than) the ``Hermitian'' condition
$\G{p}=\G{p}_+$.  A projective almost complex structure $(M,J,p)$ is
compatible if and only if $\G{p}_-^{\rm symm}=0$. Thus for a
compatible projective almost complex structure $(M,J,p)$ the
integrability condition and the ``Hermitian'' condition
$\G{p}=\G{p}_+$ agree (as in the conformal case). 

The condition $\G{p}_+(X,X)=0$ is one example of many conditions one
could impose in the projective case for which there is no conformal
analogue.

\subsection{Higher projective Invariants} \label{pinvt}

On any smooth $n$-manifold $M$ the highest exterior power of the
tangent bundle $(\Lambda^nTM)$ is a line bundle.  As discussed in
section \ref{ci}, for any smooth $n$-manifold $M$ $(\Lambda^nTM)^2$ is
an orientable line bundle. Again we choose an orientation and a root:
For our subsequent discussion it is convenient to take the positive
$(2n+2)^{th}$ root of $(\Lambda^nTM)^2$ and we denote this $K$ or
$E(1)$. Then for $w\in \mathbb{R}$ we denote $K^w$ by $E(w)$. Sections
of $E(w)$ will be described as {\em projective densities} of weight
$w$.

Now we consider a projective manifold $(M,p)$.  Each connection
$\nabla\in p$ determines a connection (also denoted $\nabla$) on
$(\Lambda^nTM)^2$ and hence on its roots $E(w)$, $w\in
\mathbb{R}$.   For $\nabla\in p$ let us (temporarily)
denote the connection induced on $E(1)$ by $D^\nabla$, and write
$-F^\nabla$ for its curvature.  
It is easily verified that, under the transformation \nn{ptrans},
$D$ transforms according to 
$$
D^{\widehat{\nabla}}_a= D^\nabla_a + \Upsilon_a,
$$
where we view $\Upsilon_a$ as a multiplication operator.
Since the connections on $E(1)$ form an affine space
modelled on $\Gamma(T^*M)$ it follows that 
by moving around in $p$ we can hit any connection on $E(1)$, and conversely 
a choice of connection on
$E(1)$ determines a connection in $p$.

Now $E(1)$ is a trivial bundle and any chosen trivialisation
determines a flat connection on $E(1)$ in the obvious way. Such a
connection will be called a {\em scale}. It follows that there is a
special class $s$ of connections in $p$: $\nabla\in s$ if and only if
$D^\nabla$ is a scale; if $D^\nabla$ is a scale we shall also call
$\nabla $ a scale.  Again using that $E(1)$ is a trivial bundle it
follows that $\nabla $ is a scale if and only if $F^\nabla=0$.

 Since 
\begin{equation}\label{curvt}
F^{\widehat{\nabla}}=F^\nabla- d \Upsilon 
\end{equation}
it is clear that if
$\nabla$ and $\widehat{\nabla}$ are both scales then $d \Upsilon=0$; in fact from the definition of scales $\Upsilon$ is then exact.
We will henceforth drop the notation $D^\nabla$ and write $\nabla$ for
any connection, induced by $\nabla\in p$, on densities, tensor bundles
and so forth. 

We now consider a projective almost complex manifold $(M,J,p)$; this
is oriented by $J$. We have then the preferred connection $\con{p}$ on
$TM$. Thus the curvature $\stack{p}{R}$ of $\con{p}$ is a (projective)
invariant of the structure. Consider now $\con{p}$ as a connection on
$E(n+1)$. From the definition of $E(n+1)$ it follows that the
curvature of $\con{p}$ on this is the trace of $\stack{p}{R}$, as given:
$$
\stack{p}{R}_{ab}{}^c{}_c~.
$$ Whence the curvature of $\con{p}$ on $K=E(1)$ is
$-\stack{p}{F}:=-F^{\con{p}}$ where $-(n+1)\stack{p}{F}=
\stack{p}{R}_{ab}{}^c{}_c$.  For comparison with other discussions of
projective geometry, such as e.g.\ \cite{BEG} we note that since
$\con{p}$ is torsion free, $\stack{p}{R}_{ab}{}^c{}_d$ satisfies the
first Bianchi identity and hence $\stack{p}{R}_{ab}{}^c{}_c=-2
\Ric_{[ab]}$ (where $[\cdots ]$ indicates the skew part). So
$(n+1)\stack{p}{F}= 2 \Ric_{[ab]}$ (and thus
$\stack{p}{F}=-\stack{p}{\beta}$, where $\stack{p}{\beta}$ is defined
in Section 3 of \cite{BEG}).

The curvature $\stack{p}{F}$ is a fundamental invariant of the
structure. By the Bianchi identity this 2-form is closed
$d\stack{p}{F}=0 $. An obvious question is whether it has
cohomological content.
Suppose we  choose a scale $\nabla$ in $p$ and compare its curvature (which
is 0) with that of $\con{p}$. Using \nn{curvt} we obtain the following.
\begin{theorem} \label{pcon} Let $(M,J,p)$ be a projective almost complex 
structure. The curvature $F^{\con{p}}$ of the canonical Weyl connection 
is an invariant of the structure.  This two form is
 exact:  for any scale $\nabla\in p$ we
 have
$$
F^{\con{p}}= d A^{\nabla}.
$$
\end{theorem}
Of course we may obtain further invariants of $(M,J,p)$ by decomposing
$F^{\con{p}}$ into its Hermitian and anti-Hermitian parts:
$$
F_{\pm}(\cdot,\cdot):= F(\cdot,\cdot) \pm F(J\cdot,J\cdot).
$$

From Theorem \ref{pcon} we see that the situation is analogous to the
conformal almost Hermitian case. Note that the analogue of conformal
transformations are {\em special projective transformations},
i.e.\ those transformations \nn{ptrans} (equivalently \nn{ptrans2})
where $\Upsilon$ is required to be exact. With these observations in mind
we see 
there is a natural hierarchy of
curvature conditions that one can consider: \\
\begin{itemize}
\item $\con{p}$, and therefore also $F^{\con{p}}$, general;
\item  $F^{\con{p}}_+=0$ or $F^{\con{p}}_-=0$.
\item $F^{\con{p}}=0$; then
\begin{enumerate}
\item $0\neq [A^{\nabla}]\in H^1(M)$, where $\nabla$ is any scale; or  
\item $A^{\nabla}$ exact, so $\con{p}$ is a scale.
\end{enumerate} 
\end{itemize}
The last of these is a projective analogue of the conformally semi-K\"ahler
(co-symplectic) condition; in case that is
satisfied then $\con{p}$ is the unique preferred scale that is
compatible with $J$.

\section{Examples} \label{EXS}

\subsection{Example of an $(M,J,c)$ structure with nonvanishing $F=\der B$ satisfying Maxwell's equations}
%based on /home/pawel/notebooks/paulandicotton/npriemann_conformal_const.nb
We consider a 4-dimensional manifold $M$, which is a local product of a real line $\bbR$ and a 3-dimensional Lie group $G$, $M=\bbR\times G$. Let $\theta^1$, $\theta ^2$, $\theta^3$ be a basis of the left invariant forms on $G$. We have 
\be\begin{aligned}
&\der\theta^1=a_3\theta^1\dz\theta^2+a_2\theta^3\dz\theta^1+a_1\theta^2\dz\theta^3\\
&\der\theta^2=b_3\theta^1\dz\theta^2+b_2\theta^3\dz\theta^1+b_1\theta^2\dz\theta^3\\
&\der\theta^3=c_3\theta^1\dz\theta^2+c_2\theta^3\dz\theta^1+c_1\theta^2\dz\theta^3,\end{aligned}\label{bia}
\ee
where $a_i,b_i,c_i$ are real constants which satisfy relations implied by the Jacobi identity $\der^2\theta^i\equiv 0$. Simple solutions for these relations are:
$$\begin{aligned}
&a_3=\frac{a_2^2c_1-a_1b_2c_1-a_1a_2c_2+a_1b_1c_2}{a_2b_1-a_1b_2}\\
&b_3=\frac{a_2b_2c_1-b_1b_2c_1+b_1^2c_2-a_1b_2c_2}{a_2b_1-a_1b_2}\\
&c_3=\frac{-b_2c_1^2+a_2c_1c_2+b_1c_1c_2-a_1c_2^2}{a_2b_1-a_1b_2},
\end{aligned}
$$
and, in the following, we will consider equations (\ref{bia}) with $a_3$, $b_3$ and $c_3$ as above.

We equip $M$ with the canonical projection $\pi:M\to G$, introduce a coordinate $t$ along the $\bbR$ factor and consider a coframe $(\omega^1=\pi^*\theta^1,\omega^2=\pi^*\theta^2,\omega^3=\pi^*\theta^3,\omega^4=\der t)$ on $M$. Using this coframe we define a $(c,J)$ structure on $M$ by representing the conformal class $c$ via
$$c\quad\ni\quad g=(\omega^1)^2+(\omega^2)^2+(\omega^3)^2+(\omega^4)^2$$
and the almost complex structure $J$ via:
$$J=\omega^1\otimes e_3-\omega^3\otimes e_1+\omega^2\otimes e_4-\omega^4\otimes e_2.$$
Here $(e_1,e_2,e_3,e_4)$ is a frame of vector fields on $M$ which is dual to the coframe $(\omega^1,\omega^1,\omega^1,\omega^1)$, $e_i\hook \omega^j=\delta^j_i$.

Now we can find the canonical connection $\con{c}$ as in (\ref{Bcon}). In particular it is easy to get the explicit formula for the $B$-form defined in (\ref{Bform}). This reads:
$$
\begin{aligned}
B\quad=\quad &\frac{a_2b_2c_1-b_1b_2c_1+b_1^2c_2-a_1b_2c_2}{2(a_2b_1-a_1b_2)}\om^1+\frac{(b_1-a_2)(a_2c_1-a_1c_2)}{2(a_2b_1-a_1b_2)}\om^2-\\
&\tfrac12 b_1\om^3-\tfrac12 b_2\om^4\\
\end{aligned}
$$
Its exterior differential is
$$F=\der B=\frac{(a_2c_1-a_1c_2)(b_2c_1-b_1c_2)}{2(a_2b_1-a_1b_2)}\om^1\dz\om^2+\tfrac12 (a_1c_2-a_2c_1)\om^3\dz\om^4$$
showing, for example, that if $ (a_1c_2-a_2c_1)\neq 0$ the structure $(M,c,J)$ \emph{not} LCAK.  

The Hodge star of $F$ is:
$$*F=\frac{(a_2c_1-a_1c_2)(b_2c_1-b_1c_2)}{2(a_2b_1-a_1b_2)}\om^3\dz\om^4+\tfrac12 (a_1c_2-a_2c_1)\om^1\dz\om^2,$$
and
$$\begin{aligned}
\der *F=&\tfrac12(b_1-a_2)(a_2c_1-a_1c_2)\om^1\dz\om^2\dz\om^3+\\
&\frac{(a_2c_1-a_1c_2)(b_2c_1-b_1c_2)(-b_2c_1^2+a_2c_1c_2+b_1c_1c_2-a_1c_2^2)}{2(a_2b_1-a_1b_2)^2}\om^1\dz\om^2\dz\om^4-\\
&\frac{c_2(a_2c_1-a_1c_2)(b_2c_1-b_1c_2)}{2(a_2b_1-a_1b_2)}\om^1\dz\om^3\dz\om^4+\frac{c_1(a_2c_1-a_1c_2)(b_2c_1-b_1c_2)}{2(a_2b_1-a_1b_2)}\om^2\dz\om^3\dz\om^4.
\end{aligned}.$$
Thus, the structure satisfies Maxwell's equations $\der *F=0$, e.g. when:
$$b_1=a_2,\quad b_2=s a_2 \quad c_2=s c_1,\quad {\rm and}\quad s={\rm const}. $$

\subsection{$(M,p,J)$ in dimension 2 with $G_-(X,X)=0$}
%based on
%/home/pawel/notebooks/paulandicotton/npriemann_projective_2_dim.nb 

Recall that an oriented 2-dimensional conformal manifold $(M,c)$ is
the same as a complex 2-manifold $(M,J)$, and indeed the same as a
conformal Hermitian manifold $(M,J,c)$.
\begin{theorem}\label{twothm}
Let $(M,p)$ be a projective structure on a 2-dimensional manifold
$M$. Assume that $(M,p)$ is compatible with the natural complex
structure $J$ on $M$ in the sense of Definition \ref{PJcomp}. Then
$\G{p}\equiv 0$. Furthermore the torsion free connection $\con{p}$
preserves the
canonical conformal class $c$ in $M$, that is $\con{p}_ig_{ij}=2
B_ig_{ij}$ for any representative $g$ of $c$.
\end{theorem}
\begin{proof}
In the following we will represent connections $\nabla$ by the
connection 1-forms $\Gamma^i_{~j}$ in a frame.

Working locally we start with a coframe $(\theta^1,\theta^2)$ on $M$,
which may be defined by means of its structure equations: \be
\der\theta^1=\alpha\theta^1\dz\theta^2,\quad\quad\der\theta^2=\beta\theta^1\dz\theta^2,\label{stre}\ee
with functions $a$ and $b$ on $M$, such that $\der^2\equiv 0$. Then we
consider connection one forms $\Gamma^i_{~j}$ defined by the equation
\be \der\theta^i+\Gamma^i_{~j}\dz\theta^j=0.\label{tof}\ee This
equation is stating that the corresponding connection $\nabla$ is
torsion free.  The relation between the 1-forms $\Gamma^i_{~j}$ and the
connection $\nabla$ is given by
$$ \Gamma^i_{~j}=\Gamma^i_{~jk}\theta^k, \quad\quad \Gamma^i_{~jk}=\theta^i(\nabla_{e_k}e_j),$$
where $(e_1,e_2)$ is the frame of vector fields on $M$ such that 
$\theta^j(e_i)=\delta^j_{~i}$. 

Given (\ref{stre}), the most general solution to the equations (\ref{tof}) is:
\be\begin{aligned}
&\Gamma^1_{~1}=a\theta^1+(\alpha+b)\theta^2\\
&\Gamma^1_{~2}=b\theta^1+c\theta^2\\
&\Gamma^2_{~1}=f\theta^1+p\theta^2\\
&\Gamma^2_{~2}=(p-\beta)\theta^1+q\theta^2
\end{aligned}
\label{cop}\ee
with arbitrary functions $a,b,c,f,p,q$ on $M$. At a given point of $M$ the values of these functions parametrise the space of torsion free connections. They define projective structures on $M$ by the following:

The projective transformations (\ref{ptrans}) rewritten in the coframe $(\theta^1,\theta^2)$ mean that the connection 1-forms $\Gamma^i_{~j}$ are in the same projective class as 
$$\hat\Gamma\phantom{}^i_{~j}=\Gamma^i_{~j}+\delta^i_{~j}\Upsilon+\Upsilon_j\theta^i,$$
where $\Upsilon$ is a 1-form, $\Upsilon=\Upsilon_1\theta^1+\Upsilon_2\theta^2$.

In two dimensions locally, and up to diffeomorphism, there are only
\emph{two} complex structures, differing by a sign. In the coframe
$(\theta^1,\theta^2)$ these can be written as:
$$J=\epsilon (\theta^1\otimes e_2-\theta^2\otimes e_1),$$
where $\epsilon=\pm 1$.  So given a projective structure $p$, represented by $\nabla$ as above (or in our frame by $\Gamma^i_{~j}$ as in (\ref{cop})), then 
 $J$ determines  the unique connection $\con{p}$ in the projective class. We will write this connection in our coframe in terms of the corresponding connection 1-forms $\gac{p}^i_{~j}$. These read:
\be\begin{aligned}
&\gac{p}^1_{~1}=(-\beta-c)\theta^1+\tfrac12(3\alpha+2b-f-q)\theta^2\\
&\gac{p}^1_{~2}=\tfrac12(\alpha+2b-f-q)\theta^1+c\theta^2\\
&\gac{p}^2_{~1}=f\theta^1+\tfrac12(-a-\beta-c+2p)\theta^2\\
&\gac{p}^2_{~2}=\tfrac12(-a-3\beta-c+2p)\theta^1+(\alpha-f)\theta^2,
\end{aligned}
\label{copg}\ee
with the $A$ as in (\ref{A}) given by
$$A=-\tfrac12(a+\beta+c)\theta^1+\tfrac12(\alpha-f-q)\theta^2.$$

Now we restrict to the situation of a projective class of connections
$\Gamma^i_{~j}$ which are compatible with $J$ in the
sense of Definition \ref{PJcomp}. This means that we impose the
restriction $\G{p}_-(X,X)\equiv 0$. It follows that, in two
dimensions, this condition is equivalent to $\G{p}\equiv 0$,
$${\rm in~two~dimensions}:\quad \G{p}_-(X,X)\equiv 0\quad \Longleftrightarrow \quad\G{p}\equiv 0.$$
Actually the $\G{p}_-(X,X)\equiv 0$ condition, in terms of the functions 
$a,b,c,f,p,q$ defined above, is equivalent to
$$c=a+\beta-2p,\quad\quad f=-\alpha-2b+q.$$
Then the connection $\con{c}$ is represented by
\be\begin{aligned}
&\gac{p}^1_{~1}=(-a-2\beta+2p)\theta^1+(2\alpha+2b-q)\theta^2\\
&\gac{p}^1_{~2}=(\alpha+2b-q)\theta^1+(a+\beta-2p)\theta^2\\
&\gac{p}^2_{~1}=-(\alpha+2b-q)\theta^1-(a+\beta-2p)\theta^2\\
&\gac{p}^2_{~2}=(-a-2\beta+2p)\theta^1+(2\alpha+2b-q)\theta^2,
\end{aligned}
\label{copgm}\ee 
and $A$ is given by
$$A=(-a-\beta+p)\theta^1+(\alpha+b-q)\theta^2.$$
Now, since $\G{p}\equiv 0$, by the general theory, the connection $\con{p}$ not only is compatible with $J$, but it actually preserves $J$. Since $J$ defines $c$, it must conformally preserve the metric $$g=g_{ij}\theta^i\theta^j=(\theta^1)^2+(\theta^2)^2,$$
representing the conformal class on $M$. 
Indeed, a short calculation shows that
$$\con{p}_ig_{jk}=2B_ig_{jk}$$
with
$$B=B_i\theta^i=(a+2\beta-2p)\theta^1+(-2\alpha-2b+q)\theta^2.$$
This finishes the proof.\end{proof}

\section{Acknowledgements}

ARG wishes to express appreciation for the hospitality of Institute of
Theoretical Physics, University of Warsaw. Similarly PN wishes to
express appreciation for the hospitality of the University of Auckland
during the preparation of this work. 

Both authors thank the organisers of the 2012 BIRS Workshop on
conformal and CR geometry, during which the article was finalised. We
also thank Paul-Andi Nagy for helpful comments during this period.

\end{document}